\documentclass{article}

\makeatletter
\def\cl@chapter{}
\makeatother
\usepackage[colorlinks=true, allcolors=blue]{hyperref}
\usepackage{amsmath,amsthm}
\usepackage{cancel}
\usepackage[nameinlink]{cleveref}

\usepackage[utf8]{inputenc}
\usepackage{amssymb,color,enumerate,graphicx,mathrsfs,algorithm,subcaption,algorithmicx,algpseudocode}
\usepackage{xspace}
\usepackage[usenames,dvipsnames]{xcolor}
\usepackage[noblocks,auth-sc]{authblk}
\usepackage[normalem]{ulem}

\usepackage{pgfplots}
\usepackage{filecontents}

\usepackage[authoryear,sort&compress,comma]{natbib}

\usepackage{rotating}

\renewcommand{\bar}{\overline}
\renewcommand{\tilde}{\widetilde}
\renewcommand{\hat}{\widehat}

\newcommand{\indexSeta}{\mathscr{A}}
\colorlet{rosso}{red!70!green}

\newcommand{\add}[1]{{#1}}
\newcommand{\remove}[1]{\textcolor{rosso}{}}

\newtheorem{Ex}{Example}

\newtheorem{theorem}[Ex]{Theorem}

\newtheorem{Def}[Ex]{Definition}

\newtheorem{remark}[Ex]{Remark}

\allowdisplaybreaks

\crefname{lemma}{Lemma}{Lemmata}
\crefname{theorem}{Theorem}{Theorems}
\crefname{prop}{Proposition}{Propositions}
\crefname{claim}{Claim}{Claims}
\crefname{algorithm}{Algorithm}{Algorithms}
\crefname{Def}{Definition}{Definitions}
\crefname{assume}{Assumption}{Assumptions}
\crefname{remark}{Remark}{Remarks}
\crefname{Ex}{Example}{Examples}
\crefname{figure}{Figure}{Figures}
\crefname{section}{Section}{Sections}
\crefname{table}{Table}{Tables}
\crefname{enumi}{Statement}{Statements}
\crefname{line}{Step}{Steps}
\crefname{equation}{}{}

\newcommand{\R}{\mathbb{R}}

\newcommand{\B}{\{0,1\}}
\newcommand{\Z}{\mathbb{Z}}

\newcommand{\ceil}[1]{\left\lceil#1\right\rceil}
\newcommand{\norm}[1]{\left \|#1\right \|}

\definecolor{mycolor}{RGB}{200,0,0}

\newcommand{\baseSet}{\mathcal{X}}
\newcommand{\benefit}{\operatorname{\tau}}
\newcommand{\ineff}{\eta} 
\newcommand{\unfair}{\phi}

\usepackage{float}
\usepackage{tikz}
\usepackage{soul}
\usetikzlibrary{graphs, graphs.standard}

\usepackage{fullpage}

\allowdisplaybreaks
\title{Fairness over Time in Dynamic Resource Allocation with an Application in Healthcare}
\author[1]{Andrea Lodi}
\author[2]{Philippe Olivier}
\author[2]{Gilles Pesant}
\author[3]{Sriram Sankaranarayanan}

\affil[1]{Jacobs Technion-Cornell Institute, Cornell Tech and Technion - IIT, New York, USA}
\affil[2]{Polytechnique  Montr\'eal, Montr\'eal, Canada}
\affil[3]{IIM Ahmedabad, Ahmedabad, India}

\date{}
\begin{document}    
\maketitle
\begin{abstract}

Decision making problems are typically concerned with maximizing efficiency. In contrast, we address problems where there are multiple stakeholders and a centralized decision maker who is obliged to decide in a fair manner. Different decisions give different utility to each stakeholder. In cases where these decisions are made repeatedly, we provide efficient mathematical programming formulations to identify both the maximum fairness possible and the decisions that improve fairness over time, for reasonable metrics of fairness. We apply this framework to the problem of ambulance allocation, where decisions in consecutive rounds are constrained. With this additional complexity, we prove structural results on identifying fair feasible allocation policies and provide a hybrid algorithm with column generation and constraint programming-based solution techniques for this class of problems. Computational experiments show that our method can solve these problems orders of magnitude faster than a naive approach. 
\end{abstract}
\section{Introduction}
\label{sec:introduction}

A resource allocation problem can be defined as the problem of choosing an allocation from a set of feasible allocations to a finite set of stakeholders to minimize some objective function.

Generally, this objective function represents the efficiency of an allocation by considering, for example, the total cost incurred by all the stakeholders~\citep{resalloc}.

However, in some cases, such an objective may be inappropriate due to ethical or moral considerations. For instance, to whom should limited medical supplies go in a time of crisis~\citep{HeierStamm2017}?
These situations call for a different paradigm of decision-making that addresses the need of {\em fairness} in resource allocation.

{Resource allocation problems trace their roots to the 1950s when the division of effort between two tasks was studied~\citep{koopman53}. Fairness in resource allocation problems has been studied since the 1960s. A mathematical model for the fair apportionment of the U.S. House of Representatives considered the minimization of the difference between the disparity of representation between any pairwise states~\citep{burt63}. This problem was revisited about two decades later~\citep{ze81, katoh85}. The so-called \emph{minimax} and \emph{maximin} objectives in resource allocation, which achieve some fairness by either minimizing the maximum~(or maximizing the minimum) number of resources allocated to entities, have been studied at least since the 1970s~\citep{jacobsen71, py72}. A recent book on equitable resource allocation presents various models and advocates a lexicographic minimax/maximin approach for fair and efficient solutions~\citep{luss2012equitable}. This work further discusses multiperiod equitable resource allocation. Fair resource allocation over time is also considered in the dynamic case where resource availability changes over time and is solved using approximation algorithms~\citep{bampis2018}. A combination of efficiency and fairness in resource allocation is common in communication networks~\citep{Ogryczak2014,Ogryczak2014-b}, and can be found in various other fields, such as in some pickup and delivery problems~\citep{Eisenhandler2019}. }

Further, theoretical computer science has also been interested in a range of fair allocation problems. 
\citet{procaccia2015cake} provides a summary of some of the principles of fair allocation of a divisible good. 
Fair allocation of indivisible good is an active area of research since the seminal paper of \citet{alkan1991fair}.
This literature consider multiple notions of fairness like (i) {\em proportionality} where each of the $n$ stakeholders receive at least $1/n$-th of the total value of all goods (ii) {\em envy-freeability}, where each stakeholder weakly prefers their own bundle over anybody else's bundle, along with other weaker notions. 
Variations of the problem where goods appear in an online fashion have also been studied \citep{aleksandrov2020online}. 
A common underlying assumption in this area is that each stakeholder gets value only based on the set of items appropriated to them. 
In our manuscript though, with the motivation being ambulance allocation, we relax this assumption. 
Even if an ambulance is not placed at the base closest to a stakeholder, but say for example, at a base that is second closest to her, she could still get limited utility from the ambulance, as opposed to the case where it is placed much farther away.  

\paragraph{Our Contributions.} {Our first contribution is to formally set up an abstract framework for repeatedly solving a fair allocation problem such that enhanced fairness is achieved over time, as opposed to a single round of allocation. Typically, a fair allocation could turn out to be an inefficient allocation. The framework addresses this trade-off by explicitly providing adequate control to the decision-maker on the level of the desired efficiency of the solutions. }

{Then, we prove that the concept of achieving fairness over time is not useful should the set of feasible allocations be a convex set, as long as the value each stakeholder obtained is a linear function of the allocation. We prove this result irrespective of the measure of fairness one uses, as long as the measure has crucial technical properties mentioned formally later. Next, we show closed-form and tight solutions for the number of rounds required for perfectly fair solutions, for special choices of the set of allocations that could be of practical interest.}

{However, there might also be other cases of practical interest where such closed-form solutions are harder, if not impossible, to obtain. We provide an integer programming formulation to solve these problems. The formulation hence provided can be combined directly with delayed column generation techniques in case of large instances. }

With the above results and ideas, we consider the problem of allocating ambulances to multiple areas of a city, where the residents of each area are the stakeholders. This family of problems is easier than the general case as there exists an efficient greedy algorithm that provides reasonable solutions fast. However, such solutions could be far from optimal, and we provide examples where the greedy algorithm fails.
Besides that, in this family of problems, we also impose restrictions on how allocations in successive rounds could be. This is reminiscent of the real-life restriction that one does not want to relocate too many ambulance vehicles each day. This added constraint makes identifying good solutions much harder as the proposed column generation-technique becomes invalid now. To solve this problem, we identify a graph induced by any candidate solution and show that deciding the feasibility of a candidate solution provided by column generation is equivalent to identifying Hamiltonian paths in this graph. Hence, we provide multiple subroutines to retain the validity of the column generation-based approach. We also provide experimental results on the computational efficiency of our method.

The paper is organized as follows. \cref{sec:prelim} introduces some basic definitions and concepts. \cref{sec:simple} presents some theoretical proofs for simple cases, and \cref{sec:finitex} establishes a general framework which allows to tackle more complicated cases. This framework is used to solve a practical problem in \cref{sec:alp}, whose results are presented in \cref{sec:numerical}. Finally, some conclusions are drawn in \cref{sec:conclusion}.

\section{Preliminaries}\label{sec:prelim}
In this section, we formally define the terms used throughout the manuscript.
We study resource allocation in the context of fairness over time, meaning that resources are allocated to stakeholders over some time horizon. 
We use $\baseSet$ to denote the set of all feasible allocations, with $n\in\Z_{\ge 0},\,n\geq 2$ being the number of stakeholders among whom the allocations should be done in a fair way.

\begin{Def}[Benefit function]
The benefit function $\benefit:\baseSet\to\R^n$ is a function such that the $i$-th component $[\benefit(x)]_i$ refers to the benefit or utility enjoyed by the $i$-th stakeholder due to the allocation $x \in\baseSet$.
\end{Def}

Assume $\unfair$ to be a function that measures the unfairness associated with a given benefit pattern for the $n$ stakeholders---a concept formalized in \cref{Def:Unfair}---which is to be minimized. In this context, a basic fair resource allocation problem can be defined simply as

\begin{align}
    \min_{x\in\baseSet} \qquad \unfair(\benefit(x)). \label{eq:intro_spfa}
\end{align}

{Many decision-makers are public servants who are either elected or nominated, and who serve in their position for a predetermined amount of time. In the context of serving the public, these officials must often juggle with limited means to ensure that everyone they serve is catered for. Public satisfaction naturally plays a role in the evaluation of their performance. Herein lies the motivation to solve resource allocation problems that consider fairness over time.}

Assuming that there are $T$ rounds of decision making, model~\cref{eq:intro_spfa} can be expanded to
\begin{subequations}
\begin{alignat}{3}
         \min_{x(t), y} \quad & \unfair(y)   \\
\textrm{s.t.} \quad & y = \frac{1}{T} \sum_{t=1}^T \benefit(x(t)) && \\
                    & x(t) \in \baseSet, &&\forall\,t = 1,\ldots,T.
                    \end{alignat}
\end{subequations}
Here, we  implicitly assume that it is sensible to add the benefits obtained at different rounds of decision making, and use the average benefit over time to compute fairness.

Now, {in this section, we first define the properties that a general measure of unfairness should have for it to be considered valid in our context. To start with, we state that (i) if all the stakeholders of the problem get the same benefit or utility, then the unfairness associated with such an allocation of benefits should be 0, and (ii) the unfairness associated with benefits should be invariant to permutations of the ordering of stakeholders. 
}
\begin{Def}[Unfairness function]\label{Def:Unfair}
Given $y\in\R^n$, $\unfair:\R^n\to\R_{\ge 0}$ determines the extent of unfairness in the allocations, if the $i$-th stakeholder gets a benefit of $y_i$. Such a function $\unfair$ satisfies the following:
\begin{enumerate}
    \item  $\unfair(y_1,\ldots,y_n) = 0 \iff y_1 = y_2 = \ldots = y_n$,
    \item $\unfair(y_1,\ldots,y_n) = \unfair(\pi_1(y),\ldots,\pi_n(y))$ for any permutation $\pi$.
\end{enumerate}
In addition, if $\unfair$ is a convex function, we call $\unfair$ a {\em convex unfairness function}.
\end{Def}

In general, trivial solutions where no benefit is obtained by any of the stakeholders, i.e., when $\benefit(x)$ is a zero vector, are perfectly fair solutions! However, this results in a gross loss of efficiency. 

\begin{Def}[Inefficiency function]
Given $\baseSet$ and the benefit function $\benefit$, the inefficiency function $\ineff:\baseSet\to[0,1]$ is defined as 
\begin{equation}
    \ineff(x) \quad=\quad \left \{
    \begin{array}{cl}
         0& \text{ if } \bar f = \underline f,  \\
    \frac{\bar{f}-\sum_{i=1}^n [\benefit(x)]_i}{\bar{f} - \underline{f}}& \text{ otherwise,}
    \end{array}
    \right .
\end{equation}
where {$\bar{f} = \sup_{x\in\baseSet}\sum_{i=1}^n [\benefit(x)]_i$, $\underline{f} = \inf_{x\in\baseSet}\sum_{i=1}^n [\benefit(x)]_i$}, and $\bar f$ and $\underline f$ are assumed to be finite.
\end{Def}

\begin{remark}
Note that for all feasible $x\in\baseSet$, we indeed have $\ineff(x) \in [0,1]$. For the most efficient $x$, i.e., the $x$ (or allocation) that maximizes the sum of benefits, $\ineff(x) = 0$, while for the least efficient $x$, we have $\ineff(x) = 1$. Thus, $\ineff$ serves as a method of normalization of the objective values. 
\end{remark}

We now define a single-period fair allocation problem, subject to efficiency constraints. One might always choose $\ineff(x)=1$ to retrieve the problem in model \cref{eq:intro_spfa} without efficiency constraints.
\begin{Def}[Single-period fair allocation (SPFA) problem]
Given $\bar \ineff \in [0,1]$, the single-period fair allocation problem is to solve 
\begin{subequations}
\label{spfa}
\begin{alignat}{3}
         \min_{x} \quad & \unfair(\benefit(x))   \\
\textrm{s.t.} \quad & 
\ineff (x) \le \bar \ineff &&.
                    \end{alignat}
\end{subequations}
\end{Def}

\cref{ex:motiv} motivates and provides a mean of validation for the SPFA. It shows that with reasonable choices, we retrieve the intuitively fair solution.

\begin{Ex}
\label{ex:motiv}
Consider the case where $\baseSet = \{x \in \R^n_+ : \sum_{i=1}^nx_i = 1\}$, i.e., $\baseSet$ is a simplex. Further, assume that $[\benefit(x)]_i = \benefit_ix_i$ for some scalars $\benefit_i > 0$, and that, w.l.o.g., $\benefit_1 \ge \benefit_2 \ge \ldots \ge \benefit_n$.

For the choice of $\bar\ineff = 1$, i.e., with no efficiency constraints, the solution for the SPFA  problem is given by
\begin{subequations}
\begin{align}
    x^\star_i \quad&=\quad \frac{g}{\benefit_i}, \quad\forall i\\
    \text{where }\quad g \quad&=\quad \frac{1}{\sum_{i=1}^n \frac{1}{\benefit_i}}.
\end{align}
Clearly, each stakeholder enjoys a benefit of $g$, and hence the unfairness associated with this allocation is $0$. We can notice that $\ineff(x^\star) = \frac{\benefit_1 - ng}{\benefit_1 - \benefit_n}$.
\end{subequations}

Note that for the case where $n=2$, the above simplifies to 
\begin{align*}
    x_1^\star \quad&=\quad \frac{\benefit_2}{\benefit_1+\benefit_2}\\
    x_2^\star \quad&=\quad \frac{\benefit_1}{\benefit_1+\benefit_2}\\
    \ineff(x^\star) \quad&=\quad \frac{\benefit_1^2 - \benefit_1\benefit_2}{\benefit_1^2 - \benefit_2^2}.
\end{align*}
\end{Ex}

\noindent We now define the fairness-over-time problem.  

\begin{Def}[$T$-period fair allocation ($T$-PFA) problem]\label{Def:TPFA}
Given $\bar \ineff \in [0,1]$ and $T \in \Z_{\ge 0}$ with $T\geq 2$, the $T$-period fair allocation problem is to solve
\begin{subequations}
\begin{alignat}{3}
         \min_{x(t), y} \quad & \unfair(y)   \\
\textrm{s.t.} \quad & y = \frac{1}{T} \sum_{t=1}^T \benefit(x(t)) && \\
& 
\ineff(x(t)) \le  \bar\ineff
, &&\quad\forall\,{t} = 1,\ldots,T \label{eq:FOT:eff} \\
& x(t) \in  \baseSet
, &&\quad\forall\,{t} = 1,\ldots,T. 
                    \end{alignat}\label{eq:FOT}
\end{subequations}
We say that a fair-allocation problem (SPFA or $T$-PFA) has {\em perfect fairness} if the optimal objective value of the corresponding optimization problems is 0. 
\end{Def}

\begin{Ex}[Usefulness of the $T$-PFA]\label{ex:tpfa_useful}
Let $\baseSet = \{x \in \B^2:x_1 + x_2 = 1\}$. Further, let $\benefit(x) = (2x_1, x_2)$. Let $\bar \ineff = 1$. Note that, in the case of the SPFA, the optimal objective is necessarily nonzero, since the benefits of the two stakeholders are unequal for every feasible solution. However, consider the 3-PFA. Now, if $x(1) = (1,0)$, $x(2) = (0,1)$, $x(3) = (0,1)$, $y = \left( \frac{2}{3},  \frac{2}{3}\right)$. So, for any choice of $\unfair$, the optimal objective value is $0$, which is strictly better than the SPFA solution.
\end{Ex}

\section{Simple Cases}
\label{sec:simple}
\subsection{Convex $\baseSet$}
We have motivated in \cref{ex:tpfa_useful} that ensuring fairness over multiple rounds could offer better results than seeking fairness in a single round. We now show that this holds only for a nonconvex $\baseSet$. In other words, if $\baseSet$ is convex and if $\benefit$ is a linear function, we necessarily get no improvement in $T$-PFA over SPFA.
\begin{theorem}\label{thm:TPFAconvex}
Let $f^\star$ and $f^\star_T$ be the optimal values of SPFA and $T$-SPFA for some nonnegative integer $T$. If $\baseSet$ is convex and if $\benefit$ is linear, then $f^\star = f^\star_T$. 
\end{theorem}
\begin{proof}
Given that $\benefit$ is a linear function, we can write $[\benefit(x)]_i = \benefit_i^{\mathsf T}x$ for an appropriately chosen $\benefit_i$, for $i=1,\ldots,n$.
Suppose that $x^\star(t)$, for $t=1,\ldots, T$, and $y^\star$ solve the $T$-PFA problem. By our notation, this has an objective value of $\unfair(y^\star) = f^\star_T$. Now, we construct a solution for the SPFA problem with an objective value equal to $f^\star_T$. 
For this, consider the point $\bar x = \frac{1}{T} \sum_{t=1}^T x^\star(t)
$. 
First we claim that $\bar x\in\baseSet$.

We have that $x(t) \in \hat{\baseSet}$ where $\hat\baseSet = \{x \in \baseSet : \ineff(x) \leq \bar\ineff\}$.
Now, we observe that $\hat\baseSet$ is a convex set given by $\{x \in \baseSet : (\sum_{i=1}^n\benefit_i)^{\mathsf T} x \geq \bar f - (\bar f - \underline f)\bar \ineff  \}$, where $\bar f$ and $\underline f$ are constants. Since $\bar x$ is obtained as a convex combination of $x^\star(t) \in \hat \baseSet$, it follows that  $\bar x \in \hat \baseSet$. 
Finally, from the linearity of $\benefit$, we can set $\bar y = y^\star$. This shows that $(\bar x, \bar y)$ is feasible. Since $\bar y = y^\star$, it follows that their objective function values are equal.\hfill $\blacksquare$
\end{proof}

Having proven that multiple rounds of allocation cannot improve the fairness when $\baseSet$ is convex and $\benefit$ is linear, we show that it is not necessarily due to the fact that perfect fairness is obtainable in SPFA. \cref{ex:lb} shows an instance where the unfairness is strictly positive with a single round of allocation, irrespective of the choice of $\unfair$. Naturally, due to \cref{thm:TPFAconvex}, perfect fairness is not possible with multiple rounds of allocation either.

\begin{Ex} \label{ex:lb}
Consider a fair allocation problem where $\baseSet = \{x\in\R^3_{\ge 0}:x_1+x_2+x_3 = 1\}$. Clearly $\baseSet$ is convex. Consider linear $\benefit$ defined as
\begin{align*}
    [\benefit(x)]_1 \quad&=\quad x_1 + \frac{3}{4} x_2 + \frac{3}{4} x_3\\
    [\benefit(x)]_2 \quad&=\quad x_2\\
    [\benefit(x)]_3 \quad&=\quad x_3.
\end{align*}
Let $\bar x$ be a fair allocation. In such a case, we need $[\benefit(\bar x)]_1 = [\benefit(\bar x)]_2 = [\benefit(\bar x)]_3$. Thus, we need 
\begin{align*}
    g \quad&=\quad x_1 + \frac{3}{4} x_2 + \frac{3}{4} x_3\\
    g \quad&=\quad x_2\\
    g \quad&=\quad x_3
\end{align*} for some $g \in \R_+$. Solving this linear system gives the {\em unique} fair allocation as allocations of the form $(x_1, x_2, x_3) = (-0.5g, g, g)$, which necessarily violates the non-negativity constraints in the definition of $\baseSet$. Any other allocation necessarily has $\unfair(\benefit(x)) > 0$.
\end{Ex}

\subsection{Simplicial $\baseSet$}

We now show that perfect fairness is attainable under certain circumstances.
In the theorem below, $LCM (\ldots)$ refers to the least common multiple of the set of integers in its arguments. 

\begin{theorem}\label{thm:TPFAsimpl}
Let $\baseSet = \{x \in \Z^n_{\ge 0}: \sum_{i=1}^nx_i \leq a\}$ for some positive integer $a$. 
Let the benefit function be $[\benefit(x)]_i = \benefit_ix_i$, where each $\benefit_i$ is a positive integer. Assume, w.l.o.g., that $\benefit_1 \ge \benefit_2 \ge \ldots \ge \benefit_n > 0$. 
Let $L = LCM(\benefit_1,\ldots,\benefit_n)$.  Then,
\begin{enumerate}
    \item \add{The only perfectly fair solution within a number of periods strictly lesser than $\bar T= \ceil{\frac{1}{a}\sum_{i=1}^n \frac{L}{\benefit_i}}$ is the trivial solution. i.e., $x=0$.}
    \remove{Perfect fairness cannot be achieved within a number of periods strictly lesser than $\bar T= \ceil{\frac{1}{a}\sum_{i=1}^n \frac{L}{\benefit_i}}$.} 
    \item Perfect fairness is achieved with $\bar T$ periods if $ \frac{\benefit_1- \min(\benefit_n,\benefit_1/a)}{\benefit_1 } \le \bar\ineff < 1
    $ and $\bar T > 1$. 
    \item Perfect fairness can be unattainable within $\bar T$ periods if $\bar\ineff < \frac{\benefit_1- \min(\benefit_n,\benefit_1/a)}{\benefit_1 }$.
\end{enumerate}
\end{theorem}

\begin{proof}
\textbf{Part 1.} Observe that an unfairness of $0$ can be achieved when for some $g\in \R$ we have $\sum_{t=1}^Tx_i(t) = g/\benefit_i$ for every $i\in\{1,\ldots,n\}$. By integrality of each $x_i(t)$ and $T$, $g$ must be a multiple of $L$, giving
\begin{align*}
    \sum_{t=1}^T x_i(t) \quad&=\quad \frac{\alpha L}{\benefit_i} \label{eq:TPFAsimplA}
\end{align*}
for some integer $\alpha \geq 0$. Because $\bar \ineff < 1$, $\alpha$ must be nonzero. Summing the above equation over $i=1,\ldots,n$, we obtain
\begin{align*}
    \sum_{t=1}^T\sum_{i=1}^n x_i(t) \quad&=\quad \sum_{i=1}^n\frac{\alpha L}{\benefit_i}.
\end{align*}
Using the fact that $\sum_{i=1}^n x_i(t) \leq a$, we obtain
\begin{align*}
    \alpha\sum_{i=1}^n\frac{ L}{\benefit_i}  \quad\le\quad \sum_{t=1}^Ta \quad=\quad aT \quad\iff\quad \frac{\alpha}{a} \sum_{i=1}^n \frac{L}{\benefit_i} \quad \le\quad T.
\end{align*}
Since $\alpha \ge 1$, the minimum LHS is attained when $\alpha=1$ and given that $T$ is integer, we obtain $$\bar T:=\ceil{\frac{1}{a}\sum_{i=1}^n\frac{L}{\benefit_i}}.$$

\textbf{Part 2. } We prove this part by exhibiting a solution feasible for $\bar T$-PFA that achieves perfect fairness over time and is hence optimal for it. 
Consider a reverse-lexicographic allocation method where all allocations are made to stakeholder $n$ until the total allocation towards $n$ sums to $L/\benefit_n$. i.e., if $L/\benefit_n \le a$, then let $x_n(1) = L/\benefit_n$. Otherwise let $x_n(1) = a$ and allow $x_n(2) = \min \{a, L/\benefit_n - a\}$. 
Repeat this process until the total allocation towards $n$ adds up to $L/\benefit_n$. 
Next, repeat the same process for $n-1$ so that the total allocation towards $n-1$ adds up to $L/\benefit_{n-1}$.
allocate similarly for $n-1$, $n-2$ and so on up to stakeholder $1$.
Observe that each of these allocations at any time $t$ is in $\baseSet$ by construction.
Again by construction, each player $i$ has a value given by $\benefit_i \times \frac{L}{\benefit_i} = L$. 
If we show that the allocation in each time period has $\ineff \le \bar\ineff$, we are done.
Consider a period $\hat t\in \{1,\ldots,\bar T\}$ where the allocation is the most inefficient. 
This would either be in the first or the last period, i.e., $\hat t=1$ or $\hat t=\bar T$. 
\add{This is because $T=1$ corresponds to adding the most possible value to stakeholder $n$, i.e., the one who values the allocation the least. Alternatively, the most inefficient allocation could be at period $T$, because in the last period the constraint $\sum_{i=1}^nx_i \le a$ might hold with a strict inequality, leading to inefficiencies.}
We will show in either case the inefficiency is at most $\bar\ineff$. 
If the first period, $\ineff$ could be large as all allocations could be to $n$, i.e., $x(1) = (0,0,\ldots,0,a)$. $\ineff$ corresponding to this allocation is $\frac{a\benefit_1- a\benefit_n}{a\benefit_1 - 0} = \frac{\benefit_1- \benefit_n}{\benefit_1} \leq \bar\ineff$. Alternatively, in the last period, there could be minimal allocation of $\sum_{i=1}^n x_i(\bar T) = 1$, but in that case, it will necessarily be made to the first stakeholder, i.e., $x(T) = (1,0,0,\ldots,0)$. 
$\ineff$ corresponding to this allocation is $\frac{a\benefit_1- \benefit_1}{a\benefit_1 - 0} = \frac{\benefit_1- \benefit_1/a}{\benefit_1} \leq \bar\ineff$.

\textbf{Part 3.} We show that perfect fairness might be unattainable with stronger efficiency requirements, by providing a family of counterexamples, each with $n=2$ stakeholders. 
Consider $\hat T \geq 2$ and
choose $\benefit_1 = (\hat T-1)a$ and  $\benefit_2 = 1$. From the first part, perfect fairness is not obtainable for $T < \bar T = \ceil{\frac{1}{a} \left( 
\frac{(\hat T-1)a}{(\hat T-1)a} + \frac{(\hat T-1)a}{1}
\right)} = 
\hat T
$ periods. We now show that perfect fairness is not possible with $\hat T$ periods either. 

Consider the  set of solutions where perfect fairness is achieved in $\hat T$ time steps.
Any such solution should award equal total utility to both stakeholders $1$ and $2$. 
The smallest possible utility is the LCM $(\bar T-1)a$, \add{as even allocating $1$ unit to player 1 already results in a utility of $\hat {T}-1$ for them}. 
To achieve this utility 
$(\hat T-1)a$ allocations are required for $2$.
Consider the way in which such an allocation can be done. 
Up to permutations across periods, They are all of the form
\begin{align*}
 x(1)\quad&=\quad (0, a - \delta_1) \\
 x(2)\quad&=\quad (0,a - \delta_2) \\
 &\vdots \\
 x(\hat T-1)\quad&=\quad (0, a - \delta_{\hat T -1}) \\
 x(\hat T)\quad&=\quad \left (\sum_{i=1}^{\hat T -1}\delta_i , 1 \right )   
\end{align*}
   for some integers $\delta_i \in \{1,2,\ldots, a\}$ satisfying $0 \leq \sum_{i=1}^{\hat T -1}\delta_i \leq  a-1$. 
   From the hypothesis, we need $\bar\ineff < \frac{\benefit_1- \min(\benefit_n,\benefit_1/a)}{\benefit_1 } = \frac{(\hat T-1)a - 1}{(\bar T-1)a}$.
   
   Now consider the inefficiency function of the allocation $x(1)$. This is minimized when $\delta_1=0$. In that case, $\ineff(x(1)) = \frac{(\hat T -1)a - 1 + \delta_1/a}{\hat T - 1}= \frac{(\hat T -1)a - 1}{(\hat T - 1)a} > \bar \ineff$ shows that it is infeasible. Thus, any perfectly fair allocation for $\bar T$-PFA problem is an infeasible allocation, providing the necessary counterexample. \hfill $\blacksquare$
\end{proof}

The first part of \cref{thm:TPFAsimpl} states that for any $n$ and a $\baseSet$ of the specified form, perfect fairness is possible after a specific, finite number of periods,  $\bar T$, provided the efficiency requirements are as stated. 
The second part shows tightness of the results saying, for any other stricter efficiency requirements, there exists a counterexample with just two stakeholders, such that perfect fairness is impossible in $\bar T$ steps. 

\begin{Ex}
Consider \cref{thm:TPFAsimpl} applied to \cref{ex:tpfa_useful}. We have $\baseSet = \{(1,0), (0,1)\}$, and thus $a=1$. Then, $\benefit(x) = (2x_1, x_2)$, so $\benefit_1=2$ and $\benefit_2=1$, and $\benefit_1 \geq \benefit_2 \geq 0$ holds. In addition, $L = LCM(2, 1) = 2$. Then, $\bar T = \ceil{\frac{1}{a}\sum_{i=1}^n \frac{L}{\benefit_i}} = \ceil{\frac{1}{1}(\frac{2}{2} + \frac{2}{1})} = 3.$
\end{Ex}

\subsection{Sparse Simplicial $\baseSet$}
The sparse simplicial form for $\baseSet$ is another interesting case from a practitioner's perspective. For example, a government might want to allot money {to} $n$ possible projects, but if it divides among all of them, then it might be insufficient for any of them. So, there could be a restriction that the government allots it to at most $r$ projects.
\begin{theorem}\label{thm:TPFAcard}
Let $\baseSet =\{x \in \R^n_{\ge 0}: \sum_{i=1}^n x_i \leq a, \norm{x}_0 \leq r \}$ where $\norm{\cdot}_0$ is the sparsity pseudo-norm, which counts the number of non-zero entries in its argument.
Assume that the benefit function is $[\benefit(x)]_i = \benefit_ix_i$, where each $\benefit_i$ is a positive integer. Assume, w.l.o.g., that $\benefit_1 \ge \benefit_2 \ge \ldots \ge \benefit_n>0$. 
We denote by $(p,q)$ the quotient and the remainder for $\frac{n}{r}$.
Let $\bar T= \ceil{\frac{n}{r}}$ and $1>\bar \ineff \ge \frac{a\benefit_1 - qv}{a\benefit_1}$ where $v = a/\max \left \{\sum_{i= (\bar T-1)r+1 }^{n} \frac{1}{\benefit_{i}}, \sum_{i=(p-1)r+1}^{pr} \frac{1}{\benefit_{i}}\right \}$. There is no perfect assignment if $T < \bar T$ whereas perfect fairness is attainable in $\bar T$ periods.
\end{theorem}
\begin{proof}
First observe that at most $r$ stakeholders can be provided value in any given period. Thus except the case where everybody gets a value of $0$, perfect fairness is impossible in any fewer than $\bar T = \ceil{\frac{n}{r}}$ periods. 

Now we show that perfect fairness is possible in $\bar T$ periods, if the conditions in the theorem statement are satisfied. Consider the following allocation: In period 
$t$ for $1\leq t\leq \bar T-1$, allocate  $\frac{v}{\benefit_i}$ for $i =  (t-1)r+1,(t-1)r+2,\ldots,tr$ and $0$ to the rest. 
For $t=\bar T$, allocate $\frac{v}{\benefit_i}$ for $i=(\bar T-1)r+1,\ldots,n$.
Note that in each period $t< \bar T$, we allocate exactly to $r$ stakeholders, and for $t=\bar T$, \add{if $q=0$, we allocate to $r$ stakeholders, else } we allocate to $q$ stakeholders, thus always satisfying $\Vert x \Vert_0 \le r$.

In any period $t<\bar T$, note that $\sum_{i=1}^n x_i(t) = \sum_{i=(t-1)r+1}^{tr} \frac{v}{\benefit_i} \le v\sum_{i=(p-1)r + 1}^{pr} \frac{1}{\benefit_i} \le v \max \left\{
\sum_{i=(\bar T-1)r+1}^{n} \frac{1}{\benefit_{i}}, \sum_{i=(p-1)r + 1}^{pr} \frac{1}{\benefit_i}
\right\} = a
$  satisfying the inequality constraint. For $t=\bar T$, we have $
\sum_{i=1}^n x_i(t) = \sum_{i=(\bar T-1)r+1}^n \frac{v}{\benefit_i} \le v \max \left\{
\sum_{i=(\bar T-1)r+1}^{n} \frac{1}{\benefit_{i}}, \sum_{i=(p-1)r + 1}^{pr} \frac{1}{\benefit_i}
\right\} = a
$, again satisfying the inequality constraint. 

We can ensure that in the first $\bar T -1$ periods, each of the allocated players receives a value of $v$, and hence the total benefit of the allocation is $rv$, giving the $\ineff(x(t)) = \frac{a\benefit_1-rv}{a\benefit_1} \leq \bar\ineff $ since $q<r$.
In the last period, we allocate a utility of $v$ to $q$ players if $q>0$ else to $r$ players. The total benefit obtained in this round is at least $qv$. Feasibility follows as before and  $\ineff(x(\bar T)) = \frac{a\benefit_1-qv}{a\benefit_1} \leq \bar\ineff$, which is feasible. Perfect fairness follows  since the benefit to each player is $v$. 
\hfill $\blacksquare$.
\end{proof}
\begin{remark}
Unlike the setting in \cref{thm:TPFAsimpl}, \cref{thm:TPFAcard} is not tight with respect to $\bar \ineff$. In other words, it remains unknown whether decreasing the allowed value of inefficiency $\bar\ineff$ still allows one to achieve perfect fairness in $\bar T$ periods. 
\end{remark}
\begin{remark}
We note that the above results, \cref{thm:TPFAconvex,thm:TPFAsimpl,thm:TPFAcard}, are agnostic to the choice of $\unfair$, and only use the fact that $\unfair(x_1,\ldots.x_m) = 0 \iff x_1 = \ldots = x_n$. Any result that holds for $\unfair(\cdot) \neq 0$ has to necessarily depend upon the choice of $\unfair$.
\end{remark}

\begin{Ex}
Consider \cref{thm:TPFAcard} applied to the following setting. $\baseSet = \{x\in\mathbb{R}_{\ge 0}^3: x_1+x_2+x_3 \le 1 \norm{x}_0 \le 2 \}$. Let $\benefit_1 = 5$ and $\benefit _2 =3$.
Now $\bar T = \ceil {\frac{n}{r}} = \ceil {\frac{3}{2}} = 2$.
One can calculate that $v=2\frac{2}{9}$, $\bar\ineff = \frac{5}{9}$.
The allocation $\left(\frac{4}{9},\frac{5}{9},0\right)$ in period 1 and $\left( 0, 0, \frac{740}{999} \right)$ in period 2, gives all players a utility of $2\frac{2}{9}$ achieving perfect fairness. 
\end{Ex}

\section{General Combinatorial $\baseSet$}
\label{sec:finitex}
In many practical areas of interest, $\baseSet$ could be more complicated than the sets presented in \cref{sec:simple}. Let $\baseSet = \{x^1,x^2,\ldots,x^k\}$. We assume that $\baseSet$ is finite, but typically has a large number of elements, given in the form of a solution to a combinatorial problem.
We consider a benefit function where $\benefit(x) = \Gamma x$, for some matrix $\Gamma$.
We thus have $[\benefit(x) ]_i = \benefit_i^{\mathsf T}x $ for $i=1,\ldots,n$.
The efficiency constraint here is trivial, as the inefficient allocations $x^j$ can be removed from $\baseSet$ to retain another (smaller) finite $\baseSet$. In fact, with linear $\benefit$, the efficiency constraint is a linear inequality of the form $\sum_{i=1}^n \benefit_i^{\top}x \geq \overline{f} - \bar\ineff (\overline{f}-\underline{f}) $.

Let $q_j$ be the {\em number of times} the allocation $x^j\in\baseSet$ is chosen over $T$ rounds of decision. In this setting, the $T$-PFA problem can be restated as\footnote{The reader may have observed that an approach based on column generation would lend itself well for such a model. This is discussed in \cref{sec:alp}.} 
\begin{subequations}
\begin{alignat}{3}
         \min_{q, y} \quad & \unfair(y)   \\
\textrm{s.t.} \quad &y_i = \frac{1}{T}\benefit_i^{\mathsf T}\left(\sum_{j=1}^kq_jx^j\right), &&\quad\forall\,i = 1,\ldots,n\\
    &\sum_{j=1}^kq_j = T&&\label{eq:fbs:t}\\
    &q_j \in \Z_{\ge 0}, &&\quad\forall\,j = 1,\ldots,k.
\end{alignat}\label{eq:finiteBaseSet1}
\end{subequations}

Note that since the theorems of \cref{sec:simple} do not extend to this case, and since \cref{ex:lb} shows that a perfectly fair solution may not even exist, we know neither the value of the fairest feasible solution, nor the smallest value of $T$ that guarantees such a solution. We now present a two-phase integer program that finds the smallest value of $T$ that guarantees the fairest feasible solution. In the first phase, we are interested in finding the fairest feasible solution. This is achieved by solving model \cref{eq:finiteBaseSet1}, replacing \cref{eq:fbs:t} with $ \sum_{j=1}^kq_j \ge 1$
to ensure that not only the fairest solution is returned, but that this solution not be $(0, 0, \ldots, 0)$. Let $\unfair^\star$ be the value of the solution found in the first phase. In the second phase, we are interested in finding the smallest $T$ which can accommodate a solution of value $\unfair^\star$. This is achieved by solving model \cref{eq:finiteBaseSet1}, again replacing \cref{eq:fbs:t} with
$
 \sum_{j=1}^k{q_j} \ge 1,
$
adding
$
 \unfair(y) = \unfair^\star,
$
and changing the objective to
$
 \min_{q,y} \sum_{j=1}^kq_j.
$

{We should note that the value of~$T^{\star}$ found by the second phase could be quite high, even for simple choices of $\baseSet$. If perfect fairness can be attained in a time horizon of size, say, 10000, and that a discrete time point represents a day, then this period of 27-odd years would probably not be suitable for most practical applications, since it is only by reaching the end of the time horizon that optimal fairness is guaranteed. In practice, then, large time horizons can be problematic. This motivates the need for consistent fairness on a smaller scale.}

{Some compromises can be made which would mitigate this issue. Manually fixing $T$ to a value not higher than the expected duration of the problem would ensure a fair distribution of resources, since reaching the end of the time horizon would be assured. Accepting some degree of unfairness would also decrease the length of the time horizon and achieve similar results. \cref{ex:fairness_delayed} illustrates these two options.}

\begin{Ex}\label{ex:fairness_delayed}
{Consider a fair allocation problem where $\baseSet = \{x\in\Z^2_{\ge 0}: x_1+x_2 = 1\}$. Further consider $\benefit$ defined as
\begin{align*}
    [\benefit(x)]_1 \quad&=\quad x_1 + \frac{15}{37} x_2\\
    [\benefit(x)]_2 \quad&=\quad x_2 + \frac{15}{47} x_1.
\end{align*}
Perfect fairness can be achieved when $T=1109$, but if we fix $T=5$, the difference between the benefits of both stakeholders will be $\sim13\%$, and if we allow a small difference of $0.1\%$ between their benefits, this could be achieved when $T=15$.}
\end{Ex}

{Another compromise would be to choose a fair way of reaching the solution. Given that we know the value of~$T^{\star}$ which grants optimal fairness, we could then identify a \emph{best path} leading to the optimal solution. A best path is the one which ensures that unfairness is kept as low as possible in the intermediate time points, without sacrificing optimality at the end of the time horizon. This would, however, require solving another linear program, which may prove burdensome if the time horizon were too large.}

{
One may think that greedily aiming for the fairest solution at each time point---while considering any unfairness incurred in previous time points, and enforcing minimal efficiency constraints by ensuring that all available resources are used at each time point---could lead to optimality at the end of the time horizon. \cref{thm:greedy} shows that it is not the case.}
\begin{theorem}\label{thm:greedy}
 {A greedy approach does not provide any guarantee of optimality for the $T$-PFA problem.}

\begin{proof}
{Consider a fair allocation problem where $\baseSet = \{x\in\Z^3_{\ge 0}:  x_1+x_2+x_3 = 1\}$. Further consider $\benefit$ defined as
\begin{align*}
    [\benefit(x)]_1 \quad&=\quad \epsilon x_1 + \frac{1}{2}x_2 + \frac{1}{2}x_3 \\
    [\benefit(x)]_2 \quad&=\quad x_2 \\
    [\benefit(x)]_3 \quad&=\quad x_3.
\end{align*}}

\noindent where $0 < \epsilon < \frac{1}{2}$.
{If $T=2$, the only perfectly fair solution~(and, as it happens, the most efficient too) is achieved by allocating the resource in turn to $x_2$ and $x_3$, over $T$. Yet, for any reasonable measure of unfairness, the greedy choice would be to allocate the resource twice to $x_1$. For instance, suppose that $\unfair$ was defined by the difference between the largest and smallest benefits, i.e., $\unfair(x) = \max_{i \in \{1, ...,n\}}[\benefit(x)]_i-\min_{i \in \{1, ...,n\}}[\benefit(x)]_i$. In such a case, the greedy choice could never deviate from a continuous allocation to $x_1$, as any other allocation would always be less fair---here, the solution would always be $T\epsilon-0=T\epsilon$.}  {Thus, a greedy approach does not necessarily lead to an optimal solution.}\hfill $\blacksquare$ 
\end{proof}
\end{theorem}

{We should note that there is a trade-off between fairness and efficiency: As illustrated in \cref{ex:fair_eff}, enforcing constraints to increase fairness will generally decrease efficiency, and vice-versa.}

\begin{Ex}\label{ex:fair_eff}
{Consider a SPFA problem where $\baseSet = \{x\in\Z^3_{\ge 0}: 1 \le x_1+x_2+x_3 \le 3\}$. Further consider $\benefit$ defined as
\begin{align*}
    [\benefit(x)]_1 \quad&=\quad x_1 + \frac{1}{2} x_2 + \frac{1}{2} x_3\\
    [\benefit(x)]_2 \quad&=\quad x_2 + \frac{1}{2} x_1\\
    [\benefit(x)]_3 \quad&=\quad x_3 + \frac{1}{2} x_1.
\end{align*}
By enforcing $\ineff(x) = 0$, the only feasible solution is $x=(3, 0, 0)$ which gives $\benefit(x)=(3, \frac{3}{2}, \frac{3}{2})$ which is not perfectly fair. In contrast, by necessitating perfect fairness, the only feasible solution is $x=(0, 1, 1)$ which gives $\benefit(x)=(1,1,1)$. This corresponds to $\ineff(x) = \frac{6-3}{6-(3/2)}=\frac{2}{3}$.}
\end{Ex}

{We should also note that the value of $T$ may have a noticeable impact on both fairness and efficiency, as illustrated in \cref{ex:t}.}

\begin{Ex}\label{ex:t}
{Consider a fair allocation problem where $\baseSet = \{x\in\Z^4_{\ge 0}: 1 \le x_1+x_2+x_3+x_4 \le 3\}$. Further consider $\benefit$ defined as
\begin{align*}
    [\benefit(x)]_1 \quad&=\quad x_1 + x_2\\
    [\benefit(x)]_2 \quad&=\quad x_2 + x_1\\
    [\benefit(x)]_3 \quad&=\quad x_3 + x_4\\
    [\benefit(x)]_4 \quad&=\quad x_4 + x_3.
\end{align*}
The smallest $T$ accommodating perfect fairness is $T=1$ with $x=(1,0,1,0)$\footnote{Or any other equivalent solution.} for $\benefit(x)=(1,1,1,1)$ and an inefficiency of $\eta(x)=\frac{6-4}{6-2}=0.5$~(with one wasted resource).
However, $T=2$ also accommodates solutions with perfect fairness, albeit with an increased efficiency. Take, for instance, $x(1)=(3,0,0,0)$ and $x(2)=(0,0,0,3)$ over two time points, for a combined $\benefit(x(1)) + \benefit(x(2))=(3,3,3,3)$ and an inefficiency of $\eta(x(t))=0$ in {\em each} period $t$. This represents a noticeable increase in efficiency over choosing the initial solution twice to cover the same time horizon.}
\end{Ex}

In other words, reaching perfect fairness in the shortest possible horizon might not lead to the most efficient solutions.

\section{Ambulance Location and Relocation}
\label{sec:alp}
The task of allocating ambulances to a set of bases in a region is a well-known problem in operations research \citep{Brotcorne2003}. The objective of this problem is generally one of efficiency: Ambulances should be allocated to prime spots such that they can quickly reach the maximum number of people.

{The definition of this problem is not unified across the literature~\citep{Gendreau2001, Brotcorne2003}.\footnote{This is due to the fact that research projects often work with datasets associated with specific regions, each of which must follow local guidelines and regulations.} Most definitions, however, agree that the population should receive an efficient service, which is generally translated into some form of coverage. In this paper, we are not solely concerned by efficiency, but we are further interested in fairness: Ambulances should be allocated such that the same set of people is not always at a disadvantage with respect to access to a quick service.}

\subsection{{Preliminary Problem Formulation}}
In this manuscript, we adopt the following definition. The region to which ambulances are allocated is divided into~$n$ demand zones, the travel time between each pair being a known quantity. Ambulances can only be allocated to bases, which are located within a subset of the zones. Each zone (equivalently, the people living in each zone) are individual stakeholders in this model. A pre-chosen $T$ rounds of decision is modeled to occur, over which the ambulances should be allocated in a fair and efficient manner. In each round, a configuration~$x$ of how $m$ ambulances are placed is chosen. 
Next, we define the benefit function for the zones. 
In this model, each zone $i$, based on its population, has a demand of $\zeta_i$ ambulance vehicles. 
A zone $i$ is said to be \emph{covered} in configuration $x$ if there are at least $\zeta_i$ ambulances located in bases from which zone $i$ could be reached in a time less than a chosen time limit called the {\em response threshold time}. 
Now, $[\benefit(x)]_i = 1$ if zone $i$ is covered by the configuration $x$ and $[\benefit(x)]_i = 0$ otherwise. We note that we use a nonlinear benefit function in this case, as opposed to the simpler cases in \cref{sec:simple} to both reflect the fact that a larger number of ambulances are essential to sufficiently serve certain regions and also to demonstrate the robustness of our methods to even certain classes of nonlinear (piecewise linear) benefit functions. The nonlinearity, as shown below, can be modeled using auxiliary variables and linear functions, to conform to the form of \cref{eq:finiteBaseSet1}.
Furthermore, efficiency constraints analogous to \cref{eq:FOT:eff} are enforced by stating that at least a fraction $f$ of the zones should be covered in each allocation. The unfairness metric $\unfair$ used is the difference between the two zones which are most often and least often covered, over  $T$, i.e., $\phi(\tau):=\min_{g,h} \{g-h: g\ge \tau_i \ge h, \forall\,i\in \{1,\ldots,n\} \}$.

With the above, the problem is a standard T-PFA problem in \cref{Def:TPFA}. However, the ambulance allocation problem involves additional constraints on allocations between two consecutive periods. Decision-makers prefer policies that ensure that not too many ambulances have to shift bases on a daily basis. Thus, between two consecutive allocations, we would ideally like to have not more than a fixed number $r$ of ambulances to shift bases. We refer to this constraint as the {\em transition constraint}.

\begin{table}[h!]
\centering
    {\begin{tabular}{rl}
       \hline
       \multicolumn{2}{c}{Variables} \\
       \hline
       $\unfair(y)$ & Unfairness associated with benefits $y$ \\
       $y_i$ & Average benefit of zone $i$ \\
       $x_i(t)$ & Number of ambulances at base $i$ on time $t$ \\
       $v_i(t)$ & Number of ambulances that can reach zone $i$ at time $t$\\
       $\benefit_{i}(t)$ & Benefit of zone $i$ at time $t$ \\
       \hline
       \hline
       \multicolumn{2}{c}{Parameters} \\
       \hline
       $\mathcal{B}$ & Set of bases \\
       $\mathcal{T}$ & Set of time points \\
       $n$ & Number of zones \\
       $m$ & Number of ambulances \\
       $a_{ji}$ & 1 if zones $i$ and $j$ are connected, 0 otherwise \\
       $\zeta_i$ & Demand of ambulances for zone $i$ \\
       $T$ & Time horizon \\
       $f$ & Fraction of zones that must be covered \\
       $r$ & Number of ambulances allowed to shift bases \\
       \hline
     \end{tabular}}
\caption{Variables and parameters of the AWT.}\label{tbl:Amb} \end{table}

Now, the problem can be formally cast as follows:
\begin{subequations}
\begin{alignat}{3}
         \min_{y,v,x,\benefit,\unfair}\quad& \unfair(y) && \label{eq:Amb:begin}    \\
        \textrm{s.t.}\quad
        &x_i(t) = 0 &&\forall\, i \not \in \mathcal{B};\, \forall\,t\in\mathscr{T}\label{eq:Amb:Base}\\
        &\sum_{i=1}^n x_i(t) \le m &&\forall\,t\in\mathscr{T} \label{eq:Amb:m}\\
        &v_i(t) = \sum_{j=1}^na_{ji}x_j(t)&&\forall\,i=1,\ldots,n;\,t\in\mathscr{T}\label{eq:Amb:v}\\
        &v_i(t) \leq (\zeta_i-1) + m\benefit_i(t)&&\forall\,i=1,\ldots,n;\,t\in\mathscr{T} \label{eq:Amb:vzeta1}\\
        &v_i(t) \geq \zeta_i - m(1-\benefit_i(t))&&\forall\,i=1,\ldots,n;\,t\in\mathscr{T}\label{eq:Amb:vzeta2}\\
        &y_i  =\frac{1}{T} \sum_{t=1}^T \benefit_i(t) &&\forall\,i=1,\ldots,n\\
        &\sum_{i=1}^{n}\benefit_i(t) \ge fn &&\forall\,t\in\mathscr{T} \label{eq:Amb:eff}\\
        &x_i(t) \in \mathbb{Z}_{\ge 0}&&\forall\,i=1,\ldots,n;\,t\in\mathscr{T} \label{eq:Amb:integer} \\
        &\benefit_i(t) \in \B&&\forall\,i=1,\ldots,n;\,t\in\mathscr{T} \label{eq:Amb:end}  \\
        &\sum_{i=1}^n\left\vert x_i(t+1) - x_i(t)\right \vert \le 2r&&\forall\,t\in1,\ldots,T-1 \label{eq:Amb:trans}
\end{alignat}\label{eq:Amb}
\noindent Here, $\mathscr {T} = \{1,\ldots,T\}$, $x_i(t)$ is the number of ambulances allotted to zone $i$ at time $t$. 
Constraints \cref{eq:Amb:Base} ensure that allocation occurs only if $i$ is an ambulance base. Here, $\mathcal B$ is the set of ambulance bases. Constraints \cref{eq:Amb:m} limit the number of ambulances available for allocation in each round. Moreover, $a_{ij}$ is a binary parameter which is $1$ if an ambulance can go from $i$ to $j$ within the response threshold time and $0$ otherwise.
 Through constraints \cref{eq:Amb:v}, $v_i(t)$ counts the number of ambulances that can reach zone $i$ within the response threshold time. 
 Here, 
$\benefit_i(t)$ is $1$ if $v_i(t)$ is at least $\zeta_i$, i.e., if the demand of ambulances in zone $i$ is satisfied, and $0$ otherwise. This is accomplished by constraints \cref{eq:Amb:vzeta1,eq:Amb:vzeta2}. 
Constraints \cref{eq:Amb:eff} take care of efficiency and ensure only allocations that cover at least a fraction $f$ of all zones are to be considered. 
Finally, constraints \cref{eq:Amb:trans} are the transition constraints, which can easily be reformulated through linear inequalities. 
It ensures that not too many ambulances shift bases between consecutive time periods. 
\end{subequations}

The problem in \cref{eq:Amb} is a mixed-integer linear program (MILP). However, the problem is symmetric with respect to certain permutations of the variables. Symmetry makes it particularly hard for modern branch-and-bound-based solvers \citep{margot2010symmetry}. Even if one relaxes the transition constraints \cref{eq:Amb:trans}, the symmetry exists.
One can observe that, relaxing the transition constraints \cref{eq:Amb:trans} in \cref{eq:Amb}, we have a T-PFA problem. 
We call this relaxed problem defined by \cref{eq:Amb:begin} to \cref{eq:Amb:end} as the Ambulance-without-transition constraints (AWT) problem. 

\subsection{{A Branch-and-Price Reformulation of the AWT}}
Since the AWT in \cref{eq:Amb:begin} to \cref{eq:Amb:end} is a T-PFA problem, it can be readily written in the form shown in \cref{eq:finiteBaseSet1} and hence can be solved using branch-and-price. 
First, we note that without the constraint in \cref{eq:Amb:trans}, any feasible solution or optimal solution remains feasible or optimal after permutations to $t$. Thus, the following version of the problem could be used to solve the relaxed problem without the symmetry. Here, we only {\em count} the number of times each configuration might be used over $T$ periods.

\begin{table}[h!]
\centering
    {\begin{tabular}{rlrl}
       \hline
       \multicolumn{2}{c}{Variables} & \multicolumn{2}{c}{Parameters} \\
       \hline
       $g$ & Largest average benefit & $\benefit_{ij}$ & Benefit of zone $i$ in configuration $j$ \\
       $h$ & Smallest average benefit & $T$ & Time horizon \\
       $y_i$ & Average benefit of zone $i$ & $n$ & Number of zones \\
       $q_j$ & Number of times configuration $j$ is used & $k$ & Number of configurations \\
       \hline
     \end{tabular}}
\caption{Variables and parameters of the MP.}\label{tbl:mp} \end{table}

\paragraph{{The Master Problem} (MP). }
\begin{subequations}
\begin{alignat}{30}
    \min \quad& g-h &&\label{eq:BAPprimal_obj}\\
    \textrm{s.t.} \quad 
    &g \ge y_i&&\quad(\alpha_i)&&\quad \forall \,i = 1, \dots, n\label{eq:BAPprimal_g}\\
    &y_i \ge h&&\quad(\beta_i)&&\quad \forall \,i = 1, \dots, n\label{eq:BAPprimal_h}\\
    &y_i = \frac{1}{T} \sum_{j=1}^k\benefit_{ij}q_j&&\quad(\lambda_i)&&\quad \forall \,i = 1, \dots, n\\
    &\sum_{j=1}^k q_j = T&&\quad(\mu) &&\label{eq:BAPprimal_cstr2}\\
    &q_j \ge 0&&&&\quad \forall \,j = 1, \dots, k\\
    &q_j \in  \Z&&&&\quad \forall \,j = 1, \dots, k.
\end{alignat}
\label{eq:BAPprimal}
\end{subequations}

\noindent where $q_j$ counts the number of times the configuration defined by $x^j$ is used. The benefit obtained by stakeholder $i$ due to the configuration $x^j$ is $\benefit_{ij}$. $y_i$ is the average benefit that stakeholder $i$ enjoys through the time horizon of planning. The objective \cref{eq:BAPprimal_obj} is to minimize the difference between the largest ($g$) and the smallest ($h$) average benefits.
Note that, once we solve the MP, an equivalent solution to the AWT could be obtained by arbitrarily considering the allocations $x^j$ for $q_j$ times in $x(1),\ldots,x(T)$ of the AWT. Similarly, given a solution to the AWT, one could immediately identify a corresponding solution to the MP.

However, since the number of configurations are typically exponentially large in $n$ and $m$, and we only might use a handful of them in a solution, we could resort to a branch-and-price approach where the MP only contains a subset of the configurations.

Referring to the continuous relaxation of MP as CMP, the the dual of CMP can be found in Appendix A.

Considering only a subset of columns in the CMP is equivalent to considering only a subset of constraints in \cref{eq:BAPdual:many}. 
Given some dual optimal solution $(\alpha^\star,\beta^\star,\lambda^\star,\mu^\star)$ to the dual of the restricted CMP, one can find the most violated constraint in \cref{eq:BAPdual:many} and include the corresponding column in the restricted CMP. 

\paragraph{{The Pricing Problem}. }
\begin{subequations}
\begin{alignat}{3}
\min \quad          & \sum\limits_{i=1}^{n} \benefit_i \lambda_i^\star                         &&                    \label{eq:pp:obj}       \\
\textrm{s.t.} \quad 
& (\benefit_{i} = 1) \Longleftrightarrow  (a_{i}^{\mathsf T} x \geq \zeta_{i}), && \quad \forall \,i = 1, \dots, n \label{eq:pp:c1} \\
& (\benefit_{i} = 0) \Longleftrightarrow  (a_{i}^{\mathsf T} x \leq \zeta_{i}-1), && \quad \forall \,i = 1, \dots, n \label{eq:pp:c0} \\
                  & \sum\limits_{i=1}^n x_i \leq m, && \quad \forall \,i \in \mathcal{B} \label{eq:pp:m} \\
                  & \sum\limits_{i=1}^n \benefit_i \geq fn && && \label{eq:pp:95} \\
                  & \benefit \in \B^n &&  && \\
                  & x \in \mathbb{Z}_{\geq 0}^n. &&  &&
\end{alignat}
\label{eq:BAPpricing}
\end{subequations}
The minimal efficiency constraints are embedded in the pricing problem. 
 The time horizon is fixed in \cref{eq:BAPprimal_cstr2}.

Constraints \cref{eq:pp:c1,eq:pp:c0} help compute the benefits to each zone, $\tau$ and can clearly be rewritten with integer variables and linear constraints. No more than $m$ ambulances may be used \cref{eq:pp:m}, and the configuration must cover at least a fraction $f$  of the zones \cref{eq:pp:95}.

The MP reformulation of the AWT in \cref{eq:Amb:begin} - \cref{eq:Amb:end}, can now be solved very efficiently using branch-and-price. However, a solution thus obtained might violate \cref{eq:Amb:trans}. Further, given the solution in the space of variables in the MP it is not immediate if one can easily verify whether the constraint \cref{eq:Amb:trans} is or is not satisfied. So, given a feasible point for the MP, we define the configuration graph as follows, and then show that the point satisfies  \cref{eq:Amb:trans} if and only if the configuration graph has a Hamiltonian path. 
\subsection{Checking \cref{eq:Amb:trans}}
In the previous section, we proposed a branch-and-price to solve the relaxed version of the AWT without including constraints \cref{eq:Amb:trans}. First, we provide an algorithm to check the feasibility of a solution provided by branch-and-price in \cref{thm:HamPath}. If it is feasible, then we are trivially done. If the solution is infeasible, we provide routines to ``cut-off'' such an infeasible solution and continue the branch-and-price algorithm. Meanwhile, we also show how using cutting planes could be impractical in the space of solutions considered in the MP or a binarized version of the problem. Then, we talk about a three-way branching scheme that could work. However, we finally resort to constraint programming, due to the limitations of commercial solvers in implementing three-way branching. 

\begin{Def}[Configuration graph]
Given a solution $\bar q$ to the MP, define $\mathcal I = \{j:\bar q_j \geq 1\}$. 
The configuration graph (CG) is defined as $G = (V, E)$, where $V = \{ (j,j'): j \in \mathcal I; j' \in \{1,2,\ldots,\bar q_j\}\}$ and $E = \{( (j_1, j_1'), (j_2, j_2') ):  \left \Vert \bar x^{j_1} - \bar x^{j_2} \right \Vert_1 \leq 2r\}$ where $\bar x^{j_1}$ and $\bar x^{j_2}$ are the configurations corresponding to the variables $\bar q_{j_1}$ and $\bar q_{j_2}$.
\end{Def}
\begin{theorem}\label{thm:HamPath}
Given a {point} $\bar q$ that is feasible to the MP, there exists a corresponding $\bar x$ that is feasible to the AWT if the CG defined by $\bar q$ contains a Hamiltonian path.
Conversely, if there is a feasible solution to the AWT which uses only the configurations with indices in  $\mathcal I = \{j:\bar q_j \geq 1\}$, then the corresponding CG has a Hamiltonian path. 
\end{theorem}
\begin{proof}
Observe that $G$ has exactly $T$ vertices, since if a configuration is found more than once in $\bar q$, it is split into as many distinct vertices in $V$, which are distinguished by the second element of the tuple. An edge $((u_1,u_2), (v_1,v_2))$ in $E$ indicates that movement between configurations indexed by $u_1$ and $v_1$ does not violate the transition constraints. A Hamiltonian path is a path that visits each vertex exactly once. As such, any Hamiltonian path in $G$ proves that a feasible sequence of transitions which does not violate the transition constraints exists between the configurations in $\bar q$. Given a Hamiltonian path $(v_1, w_1), (v_2, w_2), \dots, (v_T, w_T)$ in $G$, a feasible sequence $x(1), x(2), \dots, x(T)$ for the AWT would be $\bar x^{v_1},\bar  x^{v_2}, \dots,\bar  x^{v_T}$. 

{Conversely, given a feasible sequence $x^{j_1}, x^{j_2}, \dots, x^{j_T}$ for the AWT, a Hamiltonian path can be constructed for $G$ as follows. The $t$-th vertex of the path is ${(j_t, \beta)}$ where 
$\beta = 1 + $ the number of times the configuration denoted by $x^{j_t}$ has appeared in the first $t-1$ vertices of the path. } 
\hfill $\blacksquare$
\end{proof}

Following \cref{thm:HamPath}, one can construct the CG with exactly $T$ vertices and can check if the solution to the MP is feasible to the AWT. If yes, we are done. If not, the way we can proceed is detailed in the rest of this section. 

\subsubsection{{Cutting Planes and Binarization}}

The most natural way to eliminate a solution that does not satisfy a constraint is by adding a cutting plane. This is a common practice in the MILP literature. 

In cases where a more sophisticated cutting plane is not available, but every feasible solution is determined by a binary vector, no-good cuts could be used to eliminate infeasible solutions one by one \citep{dAmbrosio2010interval}. 
However, in our problem of interest, the variables $x$ in the MP are general integer variables. 
It is possible that a point $x$ that violates the transition constraint could lie strictly in the convex hull of solutions that satisfy the transition constraint. Hence it could be impossible to separate $x$ using a cutting plane. \cref{ex:confGr} demonstrates the above phenomenon. 

\begin{Ex}\label{ex:confGr}
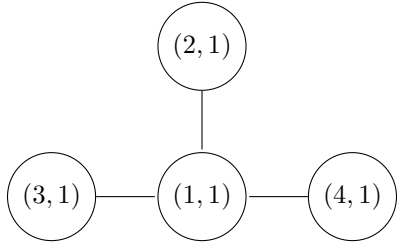
\begin{figure}[t]
    \centering
        \begin{tikzpicture}[shorten >=1pt,node distance=2cm,auto]
      \tikzstyle{state}=[shape=circle,draw,minimum size=.5cm]
      \node[state] (C) {$(1, 1)$};
      \node[state,above of=C] (A) {$(2, 1)$};
      \node[state,left of=C] (D) {$(3, 1)$};
      \node[state,right of=C] (E) {$(4,1)$};
      \path[draw]
      (A) edge node {} (C)
      (D) edge node {} (C)
      (E) edge node {} (C);
    \end{tikzpicture}
    \caption{Configuration graph for the solution in \cref{ex:confGr}.}
    \label{fig:confGr}
\end{figure}

Consider the case where the MP has an optimal solution given by $q = (q_1, q_2, q_3, q_4) = (1,1,1,1)$ and the allocations $x_1,x_2,x_3,x_4$ (in the context of the AWT) corresponding to $q_1, q_2, q_3, q_4$ are $(3,3,3,3)$, $(4,2,3,3)$, $(2,4,3,3)$, $(3,3,2,4)$ respectively. Let $r = 1$. Then, the CG corresponding to this solution is a tree as shown in \cref{fig:confGr} and hence does not have a Hamiltonian path. On the other hand, for the choices $q^i$, $q^{ii}$, $q^{iii}$ and $q^{iv}$ being $(4,0,0,0)$, $(0,4,0,0)$, $(0,0,4,0)$, $(0,0,0,4)$,  the CGs are all $K_4$, i.e., complete graphs and hence have a Hamiltonian path, trivially. 

Now, it is easy to see that $q$ lies in the convex hull of $q^i$, $q^{ii}$, $q^{iii}$ and $q^{iv}$ and while each of the latter solutions satisfies the transition constraint, $q$ does not. Thus no valid cutting plane can cut only the infeasible point. 
\end{Ex}

While \cref{ex:confGr} shows that an infeasible solution cannot always be separated using cutting planes, it could still be possible that there is an extended formulation where the infeasible point could be separated. 

\paragraph{{Naive Binarization}.} A natural choice of an extended formulation comes from binarizing each integer variable $q_j$ in the MP. This is possible because each $q_j$ is bounded above by $T$ and below by $0$. Thus one can write constraints of the form $q_j = b_j^1 + 2b_j^2 + 4b_j^4 + 8b_j^{8} + \ldots$ where the summation extends up to the smallest $\ell$ such that $2^\ell > T$, i.e., $\ell = \ceil{\log_2 T}$. If each $b_j^\ell$ is binary, any integer between $0$ and $T$ can be represented as above. 
Having written the above binarization scheme, one can separate any solution $q$ by adding a  no-good cut on the corresponding binary variables. 
A no-good cut is a linear inequality that separates a single vertex of the $0-1$ hypercube without cutting off the rest \citep{balas1972canonical,dAmbrosio2010interval}.
An example is shown in \cref{ex:confGrBin}.

\begin{Ex}\label{ex:confGrBin}
Consider the problem in \cref{ex:confGr}. Binarization to separate the solution $q = (1,1,1,1)$ can be done by adding the following constraints to the MP.
\begin{subequations}
\begin{align}
    q_j \quad&=\quad b_j^1 + 2b_j^2 + 4b_j^4 &\forall\,j=1,\ldots,4\\
    \sum_{j=1}^4\left (b_j^1 + (1-b_j^2) + (1-b_j^4) \right) \quad&\le\quad {3}
    \end{align}
\end{subequations}

Note that the second constraint above (the no-good constraint) is violated {\em only} by the binarization corresponding to the solution $u = (1,1,1,1)$ and no other feasible solution is cut off. 
\end{Ex}

The potential downside with the above scheme is that one might have to cut off a prohibitively large number of solutions before reaching the optimal solution. And with the column generation introducing new $q$-variables, this could lead to an explosion in the number of new variables as well as the number of new constraints.  

\paragraph{{Strengthened Binarization}.} When the CG corresponding to $q$ is not just lacking a Hamiltonian path, but is a disconnected graph (a stronger property holds), one could add a stronger cut, which could potentially cut off multiple infeasible solutions. In this procedure, we add a binary variable $b_j$ for each $q_j$ such that $b_j = 1$ if and only if $q_j \geq 1$. Now, observe that if a CG corresponding to the solution $\bar q$ is disconnected, then it will be disconnected for all $q$ whose non-zero components coincide with the non-zero components of $\bar q$. i.e., no solution with the same support as that of $\bar q$ could satisfy the transition constraints. Thus one could add a no-good cut on these $b_j$ binary variables, which cuts off all the solutions with the same support as $\bar q$. 

Unlike naive binarization, while this cuts off multiple solutions simultaneously, it could happen that the CG is connected, but just does not have a Hamiltonian path. In such  a case, no cut could be added by the strengthened binarization scheme, and one might have to resort to the naive version. 

\subsubsection{{Three-Way Branching}}
An alternative to the naive binarization scheme is three-way branching. While this could work as a stand-alone method, it could also go hand in hand with the strengthened binarization mentioned earlier. This method takes advantage of the three-way branching feature that some solvers, for example, SCIP \citep{GamrathEtal2020ZR,GamrathEtal2020OO}, have. 

In this method, as soon as a solution $\bar q$ satisfying all the integrality constraints but violating the transition constraint is found, the following actions are performed. First, if the lower bound and the upper bound for every component of $\bar q$ match, then we are at a leaf that can be discarded as infeasible. If not, find a component $j$ (in our case, a configuration $j$) such that $\bar q_j$ is strictly different from at least one of the bounds. Now, we do a three-way branching on the variable $q_j$ where the new constraints in each of the three branches are (i) $q_j \leq \bar q_j - 1$ (ii) $q_j = \bar q_j$ (iii) $q_j \geq \bar q_j + 1$. 

Finite termination follows from the fact that $\bar q_j$ is infeasible for branches (i) and (iii) and hence will never be visited again. For branch (ii), we have $q_j$ such that its lower and upper bounds are equal to $\bar q_j$. Thus, we have one more fixed variable and this variable will never be branched on again.

Three-way branching is used as opposed to regular two-way branching but on integer variables because of the following reasons. First, branching with (i) $q_j \leq \bar q_j - 1$  (ii) $q_j \geq \bar q_j + 1$ is invalid, as it could potentially cut off other feasible solutions with $q_j = \bar q_j$ but the solution differing in components other than $j$. Branching with (i) $q_j \leq \bar q_j - 1$ (ii) $q_j \geq \bar q_j$ could cycle, as we have not fixed any additional variable in the second branch, nor have we eliminated the infeasible solution. Thus, the LP optimum in the second branch is going to be $\bar q$ again, and cycling ensues.

\subsubsection{Constraint Programming}

The final alternative we can use to enforce transition constraints \cref{eq:Amb:trans} is constraint programming~(CP). Namely, CP is a programming paradigm for solving combinatorial problems. A CP model is defined by a set of variables, each of which is allowed values from a finite set, its \emph{domain}. The relationship between these variables is determined by the constraints of the problem. These succinct constraints can encapsulate complex combinatorial structures, such as the packing of items into bins, for example. A solver then solves the problem by enforcing consistency between the variables and the constraints, and using branching and backtracking techniques to explore the solution space. Before defining the CP model to enforce constraints \cref{eq:Amb:trans}, we define the compact configuration graph, and explain its relationship with the CG.
\begin{Def}[Compact configuration graph]
\label{def:CCG} Given a feasible solution $ q^\star$ to the continuous relaxation of the MP (CMP), the compact configuration graph (CCG) is defined as $G = (V, E)$, where $V = \indexSeta:=\{j: q^\star_j > 0\}$ and {$E = \{( v, w ):  \left \Vert  \bar x^{v} -  \bar x^{w} \right \Vert_1 \leq 2r, \forall v, w \in \indexSeta\}$}.
\end{Def}
\begin{Def}[Walk]
A \emph{walk} in an undirected graph $G=(V,E)$ is a finite sequence of vertices $v_1, \ldots, v_k$, not necessarily distinct, such that for each $i$, $(v_i, v_{i+1}) \in E$.
\end{Def}

\begin{theorem}\label{thm:walk}
Every walk of length $T$ in a CCG constructed from a solution $q^\star$ to CMP corresponds to a feasible solution of the AWT.
\end{theorem}
\begin{proof}
Let $G=(V,E)$ be the CCG given a solution $q^\star$. Let $W=v_1, v_2, \dots,v_T$ be a walk of length $T$ in $G$. Since we allow {revisiting vertices}, it is possible that {$v_j = v_{j'}$ for some $j\neq j'$}.

Now define $\tilde q$ component-wise where $\tilde q_j$ corresponds to the number of times the vertex $j$ is visited in the walk W. Since the walk has a length $T$, trivially $\sum_j \tilde q_j = T$, satisfying \cref{eq:BAPprimal_cstr2}. Then, $\tilde y, \tilde q$  can be defined so that we have a feasible {solution} to the MP. Hence, if we now show that the CG defined by the nonzero components of $\tilde q$ has a Hamiltonian path, then the corresponding $\tilde x$ will be feasible for the AWT due to \cref{thm:HamPath}.

Now, in the CG, construct the path $P = (v_1, n(v_1)+1), (v_2, n(v_2)+1), \ldots, (v_T, n(v_T)+1)$ where $n(v_j)$ records the number of times $v_j$ has appeared in the path $P$ earlier so far. We note that each term in the path $P$ is indeed a vertex of the CG (which has $T$ vertices) and that they are all visited exactly once, implying that $P$ is the required Hamiltonian path.\hfill $\blacksquare$
\end{proof}

The general mechanism involving the CP component is provided in \cref{alg:cp}, which is the entire algorithm we test in \cref{sec:numerical}.
The CP component checks whether the solution returned by the CMP can be made valid in some way, i.e., if there exist solutions using only the configurations in $\indexSeta $ (i.e., the configurations that appear in the optimal CMP solution, see Definition \ref{def:CCG}) with integer values, such that they do not violate the transition constraints. By providing these feasible solutions, it provides an upper bound to the AWT. As soon as a set of configurations is given to CP, a cut is added to the CMP, which eliminates all solutions to CMP which only consist of the configurations provided to CP. Namely, if $\indexSeta$ indices the configurations given to CP, we add a cut $\sum_{j\in\indexSeta} q_j \leq T-1$, indicating that at least one of the configurations must be outside the set indexed by $\indexSeta$.

Let $k'$ be the cardinality of $\indexSeta $. A CCG is associated with solution $q $ --- this CCG forms the basis for a deterministic finite automaton. This automaton $A$ is defined by
 a tuple $(Q, \Sigma, \delta, q_0, F)$ of states $Q$, alphabet $\Sigma$, transition function $\delta: Q \times \Sigma \rightarrow Q$, initial state $q_0 \in Q$, and final states $F \subseteq Q$, with $Q = \{0, \dots, k'\}$, $\Sigma = \{1, \dots, k'\}$, $\delta = \{(u,v) \rightarrow v : (u,v) \in E\} \cup \{(0, u) \rightarrow u : u \in F\}$, $q_0 = 0$, and $F = \{1, \dots, k'\}$. In other words, this automaton accepts any valid sequence of $k'$ configurations, with dummy configuration $q_0 = 0$ being the initial state.
  Let $LB$ be the lower bound given by the CMP, and $UB$ be an upper bound. 
 Finally, let $\Omega$ be the collection of all index sets $\indexSeta'$ for which the CP model has been previously solved. 
There are $T$ decision variables $z$ with domains 
$\{0, \dots, k'-1\}$, with $z_t$ indicating which CMP configuration from $\indexSeta $ is used at time point $t$. These variables are constrained by

\begin{subequations}
\begin{alignat}{3}
& \min \phi \label{eq:cp:obj}\\
&\phi = \texttt{max}(c)  - \texttt{min}(c) \label{eq:cp:phi} \\
& LB  \le \phi \le UB -1 \label{eq:cp:bounds} \\
& \texttt{cost\_regular}(z, A,\tau_{i\star},c_i) &&\qquad i = 1, \dots, n \label{eq:cp:regular} \\
& \texttt{at\_most}(T-1,z,\omega ) &&\qquad\forall \omega \subset \indexSeta :\exists \indexSeta' \in \Omega,\; \omega \subseteq \indexSeta' \label{eq:cp:cuts} \\
        &z_t \in \indexSeta &&\qquad t=1,\ldots,T \label{eq:cp:zdef} \\
        &c_i \in \mathbb{Z}_{\ge 0}&&\qquad i=1,\ldots,n \label{eq:cp:cdef}
\end{alignat}\label{eq:CP}
\end{subequations}

The objective~\cref{eq:cp:obj} is to minimize the unfairness, i.e., the difference between the zone which is covered the most, and that which is covered the least~\cref{eq:cp:phi}. 
Any feasible solution should be strictly better than the upper bound, and search by the CP solver can be interrupted as soon as a feasible solution is found whose objective value coincides with $LB$~\cref{eq:cp:bounds}. 
By \cref{thm:walk}, any sequence of configurations indexed by $\indexSeta $ and of length $T$ must correspond to a feasible solution of the AWT.
This requirement is enforced by the $\texttt{cost-regular}$~\citep{costregular} constraints~\cref{eq:cp:regular}: A \texttt{cost-regular} constraint holds for zone $i$ if $z$ forms a word recognized by automaton $A$ and if variable $c_i$ is equal to the coverage of zone $i$ over $T$: 
$c_i = \sum_{t=1}^T \tau_{i, z_t}$.
Note that the CP model does not require that all configurations indexed by $\indexSeta $ be used at least once: Since it checks all possible subsets of $\mathscr{A}$, the cut added to CMP 
does not remove any unexplored part of the search space. Finally, all subsets $\omega \subset \mathscr{A}$ previously explored, i.e.  being as well a subset of some index set $\indexSeta'$ in $\Omega$, are considered 
by the model using the \texttt{at\_most} constraint \cref{eq:cp:cuts}: At most $T-1$ variables in $z$ can take on values in $\omega$. The motivation for this is illustrated in \cref{ex:cp_cuts2}. 

\begin{Ex}\label{ex:cp_cuts2}
Assume that $\mathscr{A}=\{j', j'', j'''\}$. If \cref{eq:CP} had previously solved $\mathscr{A}'=\{j', j''\}$, this solution space would be explored again since the configurations 
in $\mathscr{A}$ are not constrained to be used at least once. 
Adding constraint $\texttt{at\_most}(T-1,z,\{j', j''\})$ in the current iteration avoids this.
\end{Ex}

The CP solver we use is OR-Tools, which currently does not implement the \texttt{cost-regular} constraint. We thus replaced \cref{eq:cp:phi,eq:cp:regular} with the equivalent but less efficient\footnote{In practice, the CP component of \cref{alg:cp} remains very fast, so this loss in efficiency is negligible.} reformulation
\begin{alignat}{2}
&\phi = \max\limits_{i=1}^{n} \sum\limits_{t=1}^T \tau_{i, z_t} - \min\limits_{i=1}^{n} \sum\limits_{t=1}^T \tau_{i, z_t} \nonumber \\
&\texttt{regular}(z, A). \nonumber
\end{alignat}

\subsection{The Final Algorithm}
The general working of the final algorithm is presented formally in \cref{alg:cp}.

In each iteration of the final algorithm, we solve a linear program and make a call to the constraint programming solver. The linear program is basically the continuous relaxation of the MP with any subsequent cutting planes (Step \ref{st:cuts} in \cref{alg:cp}). This is solved using column generation. The configurations which receive non-zero weights (indexed by $\indexSeta$) in the optimal LP solution are passed to the constraint programming solver to detect whether there exists a solution to the AWT that uses only configurations indexed by $\indexSeta$. Meanwhile, a cut is added to the linear program, that eliminates all solutions which use only the configurations indexed by $\indexSeta$. 

With this, the large set of possible configurations are managed by the linear programming solver, which is powerful and efficient for large problems, while the CP-based solver only solves instances with a handful of configurations at a time.

\begin{algorithm}[t]
\caption{The Final {Algorithm}}\label{alg:cp}
\begin{algorithmic}[1]
\Require The ambulance allocation problem with number of zones $n$, the bases $\mathcal B$, $\zeta_i$, for $i=1,\ldots,n$, $f$ and $a_{ih}$ for $i,h = 1,\ldots,n$,  $r\in\Z_{\ge 0}$ and $T \in \Z_+$. 
\Ensure $x(1),\ldots, x(T)$ that is optimal to the AWT.
\State $LB\gets -\infty$, $UB\gets +\infty$. $ \mathscr{C} \gets\emptyset$. $x^\star(1),\ldots,x^\star(T)=NULL$.
\While {$UB > LB$}
\State Solve CMP, i.e., the continuous relaxation of the MP after adding each constraint in $\mathscr{C}$. Let $q^\star$ be the optimal solution and $o^\star$ be the optimal objective value.
\State $LB \gets  o^\star $.
\State $\indexSeta \gets \{j:q^\star_j > 0\}$.
\State $\mathscr{C} \gets \mathscr{C}\cup \left\{ \sum\limits_{j\in\indexSeta}q_j \leq T-1. \right\}$. \label{st:cuts}
\State $(x^\dagger(1),\ldots,x^\dagger(T)), o^\dagger \gets \textsc{ConstraintProgramming}(q^\star, n, a, T, LB, UB, \mathscr{C}) $.
\If {$o^\dagger < UB$}
 \State $UB\gets o^\dagger$ and $(x^\star(1),\ldots,x^\star(T)) \gets (x^\dagger(1),\ldots,x^\dagger(T))$.
\EndIf
\EndWhile
\State \Return $x^\star(1),\ldots,x^\star(T)$
\State
\Function {ConstraintProgramming}{$q, n, a, T, LB, UB, \mathscr{C}$}
\State $G\gets$ CCG defined by $\mathscr{A}$.
\State Solve \cref{eq:CP} on the graph $G$ with appropriate values of $\benefit, c$.
\If {\cref{eq:CP} is infeasible}
\State \Return $NULL$, $+\infty$
\Else
\State \Return $(x^\dagger(1),\ldots,x^\dagger(T)), o^\dagger$.
\EndIf
\EndFunction
\end{algorithmic}
\end{algorithm}
\section{{Computational Experiments}}\label{sec:numerical}

In this section, we introduce the setting that we have used to evaluate computationally \cref{alg:cp} and we discuss the results of such an evaluation in the context of ambulance location and relocation. In particular, we use real data from the city of Utrecht, Netherlands, as well as synthetically-generated instances of varying sizes. Moreover, we consider a predetermined time horizon, i.e., a fixed $T$ of size 30, with the understanding that a decision maker could reasonably make plans on a monthly basis.

\paragraph{Utrecht instance.} We use the model defined in \cref{sec:alp} to determine ambulance allocation in the city of Utrecht, Netherlands, using a dataset provided by the RIVM.\footnote{National Institute for Public Health and the Environment of the Netherlands.} The Utrecht instance contains 217 zones, 18 of which are bases. Since calls should be reached within 15 minutes, we consider that a zone is connected to another if an ambulance can travel between the two zones in under 15 minutes. This makes the graph about 28.4\% dense.
\add{Moreover, the fleet of ambulance consists of 19 ambulance vehicles.}

The efficiency measure we impose is that at least 95\% of all the zones should be covered at all times. 
We consider that a zone is covered if sufficiently many ambulances are in the vicinity of that zone. In turn, we define the sufficient number of ambulances for a zone to be 1, 2, 3, or 4, based on the population density of the zone. 

\paragraph{Synthetic instances.} Reproducing the ratio of bases and edges to zones, we also generated synthetic instances\footnote{These instances can be found at https://github.com/PhilippeOlivier/ambulances.} of sizes 50, 100, 200, and 400 zones, using the Mat\'ern cluster point process \citep{matern}, which helps in generating a distribution of zones mimicking a realistic urban setting. The number of ambulances for the instances is chosen to be just enough to ensure feasibility. The results are averaged over five instances of each size.

\paragraph{Testing environment and software. } Testing was performed on an Intel 2.8 GHz processor with 8 GBs of RAM running Arch Linux, and the models were solved with Gurobi 9.0.1 and OR-Tools 8.0. A time limit of 1,200 seconds was imposed \add{to solve each instance}. 

\paragraph{Results.}
It is easy to see that there are two parameters that are likely to influence the difficulty of the solving process: the size of the instance (number of zones) and the transition constraints, i.e., the flexibility we allow for relocating ambulances between zones on a daily basis. We would typically expect the size of the instance to be a significant factor, but this is likely to be mitigated by the column generation approach, which tends to scale well. The effects that the transition constraints will have on the solving process, however, are less clear.
In order to assess such an effect, for each instance, we vary the number of ambulances that can shift bases on consecutive days ($r$ in constraints \cref{eq:Amb:trans}) from $0.1m$ to $m$ in increments of $0.1m$, where $m$ is the total number of ambulances in the instance. We call this ratio maximum transition, $MT$.
For each of these cases, we record the average of the time taken to solve these instances to optimality (capped at 1,200 CPU seconds). 
We record the final relative gap left for the instance, which is defined by $\frac{|UB_f-LB_f|}{UB_f}$, where $LB_f$ and $UB_f$ represent the values of the best lower and upper bounds at the time limit, respectively. We also record the initial relative gap for the instance, which is defined by $\frac{|UB_i-LB_i|}{UB_i}$, where $LB_i$ represents the value of the initial lower bound (without any cuts), and $UB_i$ represents the value of the initial, trivial upper bound.
We present the final relative gaps associated with different classes of instances in \cref{fig:results}. 
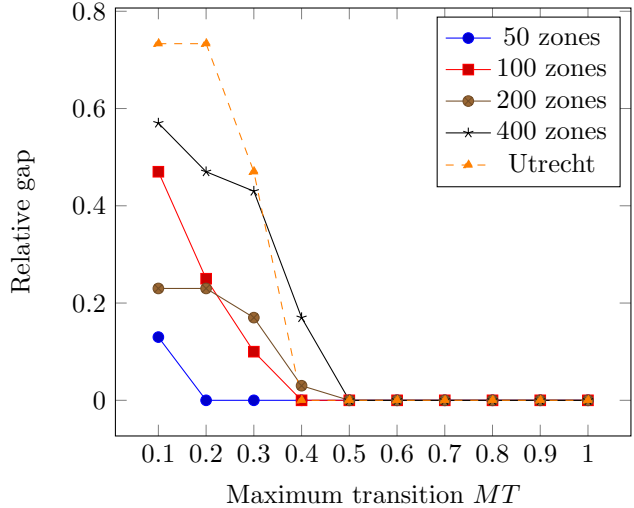
\begin{figure}[ht!]
    \centering
\begin{tikzpicture}
\begin{axis}[xtick={0.1,0.2,...,1},xlabel={Maximum transition $MT$},ylabel={Relative gap},]
\addplot table [x=mt, y=gap, col sep=comma] {50zones.csv};
\addplot table [x=mt, y=gap, col sep=comma] {100zones.csv};
\addplot table [x=mt, y=gap, col sep=comma] {200zones.csv};
\addplot table [x=mt, y=gap, col sep=comma] {400zones.csv};
\addplot[dashed, mark=triangle*, color=orange] table [x=mt, y=gap, col sep=comma] {utrecht.csv};
\legend{50 zones,100 zones,200 zones,400 zones,Utrecht}
\end{axis}
\end{tikzpicture}
    \caption{Final relative gaps with respect to the maximum transition $MT$. Gaps in the synthetic instances are averaged over 5 instances. Time limit of 1,200 CPU seconds.}
    \label{fig:results}
\end{figure}

One can immediately observe from \cref{fig:results} that the transition constraints are affecting the difficulty of the instances for $MT \leq 0.5$, while everything can be solved to optimality within the time limit for $MT \geq 0.5$ independently of the number of zones. The larger $MT$ the smaller the final gap, i.e., when there is more freedom in moving the ambulances around on a daily basis, \cref{alg:cp} becomes very effective in computing optimal solutions. 

We observe that for varying values of $MT$, the Utrecht instance has a higher relative gap than synthetic instances of comparable and even larger size. This discrepancy is the result of the distribution of the population densities across the zones. 
In the synthetic instances, the population densities are randomly distributed among the zones. 
The Utrecht instance, in contrast, exhibits several distinct clusters of varying sizes and population density distributions.\footnote{Constructing these sophisticated clusters is not a trivial task, which is why we settled on a random distribution of population densities for the synthetic instances. By randomizing the placement of the population for the Utrecht instance, its gap becomes similar to that of the synthetic instances.}

Unsurprisingly, \cref{fig:results} also suggests that larger instances are more difficult because, typically, the relative gap is large when time limits are hit. This is confirmed in more details from the results in \cref{tab:times}, where we report, for every group of instances and every $MT$ value, the initial gap (``i.gap"), the final gap at the time limit (``f.gap"),
the number of generated columns (``\#cols"), the number of solved instances (``\#solved"),\footnote{The entry ``\#solved" takes an integer value between 0 and 5 for synthetic instances and 0/1 for the Utrecht instance.} and the average time for the instances solved to optimality (``time"). 
\begin{table}[h!]
\tiny
\tabcolsep=2.5pt
    {\begin{tabular}{r|rrrrr|rrrrr|rrrrr}
       \multicolumn{1}{c|}{} & \multicolumn{5}{c|}{\textbf{50 zones}} & \multicolumn{5}{c|}{\textbf{100 zones}} & \multicolumn{5}{c}{\textbf{200 zones}} \\
       \hline
       $MT$ & i.gap & f.gap & \#cols & \#solved & time & i.gap & f.gap & \#cols & \#solved & time & i.gap & f.gap & \#cols & \#solved & time\\
       \hline
       0.1 & 0.13 & 0.13  &  738 & 4 & 0.3 &  0.55 & 0.47  & 1120 & 1 & 0.2 & 0.23 & 0.23  & 714 & 3 &  2.5 \\ 
       0.2 & 0.13 & 0.00  &  3   & 5 & 0.4 &  0.55 & 0.25  & 674  & 2 & 1.2 & 0.23 & 0.23  & 710 & 3 &  2.5 \\ 
       0.3 & 0.13 & 0.00  &  3   & 5 & 0.4 &  0.55 & 0.10  & 11   & 4 & 2.5 & 0.23 & 0.17  & 537 & 3 &  2.6 \\ 
       0.4 & 0.13 & 0.00  &  3   & 5 & 0.4 &  0.55 & 0.00  & 9    & 5 & 2.7 & 0.23 & 0.02  & 108 & 4 & 82.7 \\  
       0.5 & 0.13 & 0.00  &  3   & 5 & 0.4 &  0.55 & 0.00  & 9    & 5 & 2.3 & 0.23 & 0.00  & 19  & 5 &  7.0 \\  
       0.6 & 0.13 & 0.00  &  3   & 5 & 0.4 &  0.55 & 0.00  & 9    & 5 & 2.1 & 0.23 & 0.00  & 15  & 5 &  7.2 \\    
       0.7 & 0.13 & 0.00  &  3   & 5 & 0.4 &  0.55 & 0.00  & 9    & 5 & 2.1 & 0.23 & 0.00  & 15  & 5 &  6.7 \\    
       0.8 & 0.13 & 0.00  &  3   & 5 & 0.4 &  0.55 & 0.00  & 9    & 5 & 2.2 & 0.23 & 0.00  & 15  & 5 &  6.0 \\    
       0.9 & 0.13 & 0.00  &  3   & 5 & 0.4 &  0.55 & 0.00  & 9    & 5 & 2.2 & 0.23 & 0.00  & 15  & 5 &  6.0 \\    
       1   & 0.13 & 0.00  &  3   & 5 & 0.4 &  0.55 & 0.00  & 9    & 5 & 2.1 & 0.23 & 0.00  & 15  & 5 &  6.0 \\   
       \hline
     \end{tabular}}
    {\begin{tabular}{r|rrrrr|rrrrr}
       \multicolumn{11}{c}{} \\
       \multicolumn{1}{c|}{} & \multicolumn{5}{c|}{\textbf{400 zones}} & \multicolumn{5}{c}{\textbf{Utrecht}} \\
       \hline
       $MT$ & i.gap & f.gap & \#cols & \#solved & time & i.gap & f.gap & \#cols & \#solved & time\\
       \hline
       0.1  & 0.57 & 0.57 & 248 & 1 &  2.7  & 0.73 & 0.73 & 786  & 0 & - \\ 
       0.2  & 0.57 & 0.47 & 246 & 1 &  2.7  & 0.73 & 0.73 & 803  & 0 & - \\ 
       0.3  & 0.57 & 0.43 & 238 & 1 &  2.7  & 0.73 & 0.47 & 757  & 0 & - \\ 
       0.4  & 0.57 & 0.17 & 174 & 3 & 147.4 & 0.73 & 0.00 & 353  & 1 & 315.6 \\  
       0.5  & 0.57 & 0.00 & 93  & 5 & 273.6 & 0.73 & 0.00 & 252  & 1 & 122.3 \\  
       0.6  & 0.57 & 0.00 & 65  & 5 & 63.6  & 0.73 & 0.00 & 96   & 1 & 49.6 \\    
       0.7  & 0.57 & 0.00 & 60  & 5 & 49.6  & 0.73 & 0.00 & 96   & 1 & 50.3 \\    
       0.8  & 0.57 & 0.00 & 60  & 5 & 47.1  & 0.73 & 0.00 & 96   & 1 & 48.4 \\    
       0.9  & 0.57 & 0.00 & 60  & 5 & 47.4  & 0.73 & 0.00 & 96   & 1 & 49.8 \\    
       1    & 0.57 & 0.00 & 60  & 5 & 49.1  & 0.73 & 0.00 & 96   & 1 & 54.7 \\    
       \hline

     \end{tabular}}
\caption{Detailed results for \cref{alg:cp} on both synthetic and real-world Utrecht instances. Time limit of 1,200 CPU seconds. The results of the synthetic instances are averaged over 5 instances.} 
    \label{tab:times}
\end{table}

The results in \cref{tab:times} show that for $MT > 0.7$ constraints \cref{eq:Amb:trans} are not binding, i.e., they do not cut off the obtained optimal solution to the rest of the model, and \cref{alg:cp} behaves exactly like for $MT = 0.7$. 
The average time for the instances solved to optimality is often low, thus indicating that if we manage to prove optimality by completely closing the gap, then this is achieved quickly, generally with a few calls to \textsc{ConstraintProgramming}. In contrast, if we do not succeed in closing the gap early, then it is unlikely to be closed at all in the end. In the cases where optimality is not proven, improvements are still made during the solving process mostly because of improvements on the upper bound value as \textsc{ConstraintProgramming} tests more and more combinations of configurations. This suggests a few things. First, with a high enough $MT$, the initial lower bound is generally optimal, and feasible solutions whose objective values coincide with this bound can readily be found. Second, when the $MT$ value is low, solutions close to the initial LB are harder to find, and the gap is harder to close. Third, the computational power of \cref{alg:cp} is associated with the effectiveness in searching for primal solutions (upper bound improvement) versus the progress on the dual side (lower bound improvement), which appears to be very difficult.

\begin{table}[h!]
\tiny
\tabcolsep=2.5pt
    {\begin{tabular}{r|rrrr|rrrr|rrrr|rrrr|rrrr}
       \multicolumn{1}{c|}{} & \multicolumn{4}{c|}{\textbf{50 zones}} & \multicolumn{4}{c|}{\textbf{100 zones}} & \multicolumn{4}{c}{\textbf{200 zones}} & \multicolumn{4}{c}{\textbf{400 zones}} & \multicolumn{4}{c}{\textbf{Utrecht}} \\
       \hline
       $MT$ & $g$ & $h$ & $g-h$ & cov & $g$ & $h$ & $g-h$ & cov & $g$ & $h$ & $g-h$ & cov & $g$ & $h$ & $g-h$ & cov & $g$ & $h$ & $g-h$ & cov \\
       \hline
	   0.1 & 30 &     24 &      6 &   99.20\% & 30 &     10 &     20 &   96.20\% & 30 &      6 &     24 &   97.33\% & 30 &      6 &     24 &  98.45\% & - & - & - & - \\
	   0.2 & 30 &     28 &      2 &   99.47\% & 30 &     16 &     14 &   96.20\% & 30 &      6 &     24 &   97.33\% & 30 &     12 &     18 &  98.45\% & - & - & - & - \\
	   0.3 & 30 &     28 &      2 &   99.47\% & 30 &     20 &     10 &   96.53\% & 30 &      9 &     21 &   97.33\% & 30 &     16 &     14 &  98.45\% & 30 & 15 & 15 & 97.23\% \\
	   0.4 & 30 &     28 &      2 &   99.47\% & 30 &     23 &      7 &   96.43\% & 30 &     13 &     17 &   97.33\% & 30 &     22 &      8 &  98.19\% & 30 & 22 & 8 & 95.39\% \\
	   0.5 & 30 &     28 &      2 &   99.47\% & 30 &     23 &      7 &   96.43\% & 30 &     13 &     17 &   97.15\% & 30 &     23 &      7 &  97.23\% & 30 & 22 & 8 & 95.48\% \\
	   0.6 & 30 &     28 &      2 &   99.47\% & 30 &     23 &      7 &   96.43\% & 30 &     13 &     17 &   97.35\% & 30 &     23 &      7 &  97.21\% & 30 & 22 & 8 & 95.39\% \\
	   0.7 & 30 &     28 &      2 &   99.47\% & 30 &     23 &      7 &   96.43\% & 30 &     13 &     17 &   97.00\% & 30 &     23 &      7 &  97.23\% & 30 & 22 & 8 & 95.39\% \\
	   0.8 & 30 &     28 &      2 &   99.47\% & 30 &     23 &      7 &   96.43\% & 30 &     13 &     17 &   97.22\% & 30 &     23 &      7 &  97.13\% & 30 & 22 & 8 & 95.39\% \\
	   0.9 & 30 &     28 &      2 &   99.47\% & 30 &     23 &      7 &   96.43\% & 30 &     13 &     17 &   97.22\% & 30 &     23 &      7 &  97.23\% & 30 & 22 & 8 & 95.39\% \\
	   1   & 30 &     28 &      2 &   99.47\% & 30 &     23 &      7 &   96.43\% & 30 &     13 &     17 &   97.22\% & 30 &     23 &      7 &  97.17\% & 30 & 22 & 8 & 95.39\% \\
       \hline
     \end{tabular}}
	 \caption{Enforcing coverage constraints at 95\%. Solution values for both synthetic and real-world Utrecht instances: $g$, $h$, $g-h$ (defined in \cref{eq:BAPprimal}), and ``cov'' (the \remove{sum of all}\add{average} coverage, i.e., the efficiency). The results of the synthetic instances are averaged over 5 instances.} 
    \label{tab:gh95}
\end{table}

\begin{table}[h!]
\tiny
\tabcolsep=2.5pt
    {\begin{tabular}{r|rrrr|rrrr|rrrr|rrrr|rrrr}
       \multicolumn{1}{c|}{} & \multicolumn{4}{c|}{\textbf{50 zones}} & \multicolumn{4}{c|}{\textbf{100 zones}} & \multicolumn{4}{c}{\textbf{200 zones}} & \multicolumn{4}{c}{\textbf{400 zones}} & \multicolumn{4}{c}{\textbf{Utrecht}} \\
       \hline
       $MT$ & $g$ & $h$ & $g-h$ & cov & $g$ & $h$ & $g-h$ & cov & $g$ & $h$ & $g-h$ & cov & $g$ & $h$ & $g-h$ & cov & $g$ & $h$ & $g-h$ & cov \\
       \hline
	   0.1 & 30 &     24 &      6 &   98.80\% & 30 &     15 &     15 &   95.27\% & 30 &      6 &     24 &   94.40\% & 30 &      9 &     21 &  98.45\% & - & - & - & - \\
	   0.2 & 30 &     29 &      1 &   98.67\% & 30 &     23 &      7 &   95.10\% & 30 &      6 &     24 &   94.40\% & 30 &     13 &     17 &  98.45\% & - & - & - & - \\
	   0.3 & 30 &     29 &      1 &   98.67\% & 30 &     23 &      7 &   94.17\% & 30 &     12 &     18 &   94.40\% & 30 &     15 &     15 &  98.45\% & 30 & 19 & 11 & 90.32\% \\
	   0.4 & 30 &     29 &      1 &   98.67\% & 30 &     24 &      6 &   94.17\% & 30 &     20 &     10 &   94.30\% & 30 &     23 &      7 &  97.70\% & 30 & 22 & 8 & 90.32\% \\
	   0.5 & 30 &     29 &      1 &   98.67\% & 30 &     24 &      6 &   94.17\% & 30 &     21 &      9 &   93.33\% & 30 &     25 &      5 &  95.31\% & 30 & 23 & 7 & 91.90\% \\
	   0.6 & 30 &     29 &      1 &   98.67\% & 30 &     24 &      6 &   94.17\% & 30 &     21 &      9 &   93.65\% & 30 &     25 &      5 &  95.42\% & 30 & 23 & 7 & 91.84\% \\
	   0.7 & 30 &     29 &      1 &   98.67\% & 30 &     24 &      6 &   94.17\% & 30 &     21 &      9 &   93.30\% & 30 &     25 &      5 &  95.19\% & 30 & 23 & 7 & 91.53\% \\
	   0.8 & 30 &     29 &      1 &   98.67\% & 30 &     24 &      6 &   94.17\% & 30 &     21 &      9 &   93.30\% & 30 &     25 &      5 &  95.05\% & 30 & 23 & 7 & 91.97\% \\
	   0.9 & 30 &     29 &      1 &   98.67\% & 30 &     24 &      6 &   94.17\% & 30 &     21 &      9 &   93.28\% & 30 &     25 &      5 &  95.10\% & 30 & 23 & 7 & 92.00\% \\
	   1   & 30 &     29 &      1 &   98.67\% & 30 &     24 &      6 &   94.17\% & 30 &     21 &      9 &   93.30\% & 30 &     25 &      5 &  95.07\% & 30 & 23 & 7 & 92.00\% \\
       \hline
     \end{tabular}}
	 \caption{Enforcing coverage constraints at 90\%. Solution values for both synthetic and real-world Utrecht instances: $g$, $h$, $g-h$ (defined in \cref{eq:BAPprimal}), and ``cov'' (the \remove{sum of all}\add{average} coverage, i.e., the efficiency). The results of the synthetic instances are averaged over 5 instances.} 
    \label{tab:gh90}
\end{table}

\begin{table}[h!]
\tiny
\tabcolsep=2.5pt
    {\begin{tabular}{r|rrrr|rrrr|rrrr|rrrr|rrrr}
       \multicolumn{1}{c|}{} & \multicolumn{4}{c|}{\textbf{50 zones}} & \multicolumn{4}{c|}{\textbf{100 zones}} & \multicolumn{4}{c}{\textbf{200 zones}} & \multicolumn{4}{c}{\textbf{400 zones}} & \multicolumn{4}{c}{\textbf{Utrecht}} \\
       \hline
       $MT$ & $g$ & $h$ & $g-h$ & cov & $g$ & $h$ & $g-h$ & cov & $g$ & $h$ & $g-h$ & cov & $g$ & $h$ & $g-h$ & cov & $g$ & $h$ & $g-h$ & cov \\
       \hline
	   0.1 & 30 & 24 & 6 & 97.60\% & 30 & 10 & 20 & 93.60\% & 30 &  6 & 24 & 5616\% & 30 & 10 & 20 &  96.45\% & - & - & - & - \\
	   0.2 & 30 & 29 & 1 & 98.40\% & 30 & 24 &  6 & 92.57\% & 30 &  6 & 24 & 5616\% & 30 & 13 & 17 &  96.45\% & - & - & - & - \\
	   0.3 & 30 & 29 & 1 & 98.40\% & 30 & 25 &  5 & 92.90\% & 30 & 14 & 16 & 5616\% & 30 & 17 & 13 &  96.45\% & 30 & 22 & 8 & 94.47\% \\
	   0.4 & 30 & 29 & 1 & 98.40\% & 30 & 25 &  5 & 92.57\% & 30 & 20 & 10 & 5610\% & 30 & 22 &  8 &  96.45\% & 30 & 25 & 5 & 90.48\% \\
	   0.5 & 30 & 29 & 1 & 98.40\% & 30 & 25 &  5 & 92.57\% & 30 & 22 &  8 & 5477\% & 30 & 24 &  6 &  92.53\% & 30 & 25 & 5 & 90.48\% \\
	   0.6 & 30 & 29 & 1 & 98.40\% & 30 & 25 &  5 & 92.57\% & 30 & 22 &  8 & 5464\% & 30 & 25 &  5 &  92.52\% & 30 & 25 & 5 & 90.48\% \\
	   0.7 & 30 & 29 & 1 & 98.40\% & 30 & 25 &  5 & 92.57\% & 30 & 22 &  8 & 5448\% & 30 & 25 &  5 &  93.19\% & 30 & 25 & 5 & 90.48\% \\
	   0.8 & 30 & 29 & 1 & 98.40\% & 30 & 25 &  5 & 92.57\% & 30 & 22 &  8 & 5454\% & 30 & 25 &  5 &  93.14\% & 30 & 25 & 5 & 90.48\% \\
	   0.9 & 30 & 29 & 1 & 98.40\% & 30 & 25 &  5 & 92.57\% & 30 & 22 &  8 & 5453\% & 30 & 25 &  5 &  93.19\% & 30 & 25 & 5 & 90.48\% \\
	   1   & 30 & 29 & 1 & 98.40\% & 30 & 25 &  5 & 92.57\% & 30 & 22 &  8 & 5453\% & 30 & 25 &  5 &  93.14\% & 30 & 25 & 5 & 90.48\% \\
       \hline
     \end{tabular}}
	 \caption{Enforcing coverage constraints at 85\%. Solution values for both synthetic and real-world Utrecht instances: $g$, $h$, $g-h$ (defined in \cref{eq:BAPprimal}), and ``cov'' (the \remove{sum of all}\add{average} coverage, i.e., the efficiency). The results of the synthetic instances are averaged over 5 instances.} 
    \label{tab:gh85}
\end{table}

\cref{tab:gh95,tab:gh90,tab:gh85} show the actual values of the solutions, as well as the associated coverages (the sum of all the zones' coverages over the time horizon, i.e., the efficiency). In \cref{tab:gh95} the coverage constraints are enforced at 95\%, i.e., all the configurations must cover at least 95\% of the zones. In \cref{tab:gh90} those constraints are enforced at 90\%, and in \cref{tab:gh85} at 85\%. Comparing the three tables, we can see a clear correlation between fairness and efficiency in virtually all instances: Enforcing a higher coverage decreases the fairness of the solutions.

Finally, we are also interested in identifying the time spent in the column generation part as opposed to the \textsc{ConstraintProgramming} routine. To this end, we track the number of calls made to the function \textsc{ConstraintProgramming} and also measure the total time spent in calls to the function. 
The number of generated columns and the number of calls to \textsc{ConstraintProgramming} are detailed in \cref{tbl:results}. 

\setlength{\tabcolsep}{4.8pt}
\begin{table}[h!]
\centering
    {\begin{tabular}{crrrrrrrrrr}
       \hline
       \multicolumn{1}{c}{} & \multicolumn{2}{c}{\textbf{50 zones}} & \multicolumn{2}{c}{\textbf{100 zones}} & \multicolumn{2}{c}{\textbf{200 zones}} & \multicolumn{2}{c}{\textbf{400 zones}} & \multicolumn{2}{c}{\textbf{Utrecht}} \\
       \hline
       $MT$ & \#cols & \#calls & \#cols & \#calls & \#cols & \#calls & \#cols & \#calls & \#cols & \#calls \\
       \hline
       0.1 & 738 & 736 & 1120 & 1108 & 714 & 608 & 248 & 72& 786  & 210  \\ 
       0.2 & 3   &   1 & 674 & 662   & 710 & 605 & 246 & 72& 803  & 217  \\ 
       0.3 & 3   &   1 & 11  & 3     & 537 & 435 & 238 & 66& 757  & 202  \\ 
       0.4 & 3   &   1 & 9   & 1     & 108 & 58  & 174 & 22& 353  & 54  \\  
       0.5 & 3   &   1 & 9   & 1     & 19  & 3   & 93  & 6 & 252  & 21  \\  
       0.6 & 3   &   1 & 9   & 1     & 15  & 1   & 65  & 2 & 96  & 1  \\    
       0.7 & 3   &   1 & 9   & 1     & 15  & 1   & 60  & 1 & 96  & 1  \\    
       \hline
     \end{tabular}}
\caption{Number of columns and calls to \textsc{ConstraintProgramming} associated with the maximum transition (MT) constraints of the various instance sizes. 
}\label{tbl:results} \end{table}

The results in \cref{tbl:results} show that when enough freedom in the day-to-day movement of ambulances is allowed, optimality can generally be proven immediately, with just one call to the \textsc{ConstraintProgramming} routine. We also notice that the ratio of columns to the number of calls to the routine increases with the instance size. Such a ratio for the Utrecht instance is disproportionately high, due again to the distribution of the population densities. 

Finally, \cref{tbl:cg_cp} shows the average amount of time spent generating columns (``CG"), that spent enforcing the transition constraints through \textsc{ConstraintProgramming} (``CP") and the ratio between these two CPU times (``ratio") for the most difficult case, i.e., $MT=0.1$.
\begin{table}[h!]
\centering
    {\begin{tabular}{r|rr|r}
       \hline
                & \multicolumn{2}{c|}{time} & \\
       Instance & CG & CP & ratio \\
       \hline
50 zones  & 86  & 154 & 0.56 \\
100 zones & 411 & 550 & 0.75 \\
200 zones & 331 & 151 & 2.19 \\
400 zones & 711 & 253 & 2.81 \\
Utrecht  & 1071 & 129 & 7.96 \\
       \hline
     \end{tabular}}
\caption{CPU times in seconds spent in column generation and \textsc{ConstraintProgramming}, with $MT=0.1$.}
    \label{tbl:cg_cp}
 \end{table}

\section{Conclusion}
\label{sec:conclusion}

We introduced an abstract framework for solving a sequence of fair allocation problems, such that fairness is achieved over time. For some relevant special cases, we have been able to give theoretical proofs for the time horizon required for perfect fairness. We described a general integer programming formulation for this problem, as well as a formulation based on column generation and constraint programming. This latter formulation can be used in a practical context, as shown by the ambulance location problem applied to the city of Utrecht. The largest synthetic instances suggest that this formulation would scale well to regions twice the size of Utrecht, given that the freedom of movement of the ambulances is not overly restricted.

\section*{Acknowledgements}

The authors would like to thank the National Institute for Public Health and the Environment of the Netherlands (RIVM) for access to the data of their ambulance service. We are grateful to the anonymous referees for useful comments and remarks.

\bibliographystyle{plainnat}
\bibliography{references}

\begin{thebibliography}{25}
\providecommand{\natexlab}[1]{#1}
\providecommand{\url}[1]{\texttt{#1}}
\expandafter\ifx\csname urlstyle\endcsname\relax
  \providecommand{\doi}[1]{doi: #1}\else
  \providecommand{\doi}{doi: \begingroup \urlstyle{rm}\Url}\fi

\bibitem[Aleksandrov and Walsh(2020)]{aleksandrov2020online}
Martin Aleksandrov and Toby Walsh.
\newblock Online fair division: A survey.
\newblock In \emph{Proceedings of the AAAI Conference on Artificial
  Intelligence}, volume~34, pages 13557--13562, 2020.

\bibitem[Alkan et~al.(1991)Alkan, Demange, and Gale]{alkan1991fair}
Ahmet Alkan, Gabrielle Demange, and David Gale.
\newblock Fair allocation of indivisible goods and criteria of justice.
\newblock \emph{Econometrica: Journal of the Econometric Society}, pages
  1023--1039, 1991.

\bibitem[Balas and Jeroslow(1972)]{balas1972canonical}
Egon Balas and Robert Jeroslow.
\newblock Canonical cuts on the unit hypercube.
\newblock \emph{SIAM Journal on Applied Mathematics}, 23\penalty0 (1):\penalty0
  61--69, 1972.

\bibitem[Bampis et~al.(2018)Bampis, Escoffier, and Mladenovic]{bampis2018}
Evripidis Bampis, Bruno Escoffier, and Sasa Mladenovic.
\newblock Fair resource allocation over time.
\newblock In \emph{Proceedings of the 17th International Conference on
  Autonomous Agents and MultiAgent Systems}, AAMAS ’18, page 766–773,
  Richland, SC, 2018. International Foundation for Autonomous Agents and
  Multiagent Systems.

\bibitem[Brotcorne et~al.(2003)Brotcorne, Laporte, and Semet]{Brotcorne2003}
Luce Brotcorne, Gilbert Laporte, and Frédéric Semet.
\newblock Ambulance location and relocation models.
\newblock \emph{European Journal of Operational Research}, 147\penalty0
  (3):\penalty0 451--463, 2003.
\newblock ISSN 0377-2217.
\newblock \doi{https://doi.org/10.1016/S0377-2217(02)00364-8}.
\newblock URL
  \url{http://www.sciencedirect.com/science/article/pii/S0377221702003648}.

\bibitem[Burt and Harris(1963)]{burt63}
Oscar~R. Burt and Curtis~C. Harris.
\newblock Letter to the editor—apportionment of the u.s. house of
  representatives: A minimum range, integer solution, allocation problem.
\newblock \emph{Operations Research}, 11\penalty0 (4):\penalty0 648--652, 1963.
\newblock \doi{10.1287/opre.11.4.648}.
\newblock URL \url{https://doi.org/10.1287/opre.11.4.648}.

\bibitem[D’Ambrosio et~al.(2010)D’Ambrosio, Frangioni, Liberti, and
  Lodi]{dAmbrosio2010interval}
Claudia D’Ambrosio, Antonio Frangioni, Leo Liberti, and Andrea Lodi.
\newblock On interval-subgradient and no-good cuts.
\newblock \emph{Operations Research Letters}, 38\penalty0 (5):\penalty0
  341--345, 2010.

\bibitem[Demassey et~al.(2006)Demassey, Pesant, and Rousseau]{costregular}
Sophie Demassey, Gilles Pesant, and Louis-Martin Rousseau.
\newblock A cost-regular based hybrid column generation approach.
\newblock \emph{Constraints}, 11\penalty0 (4):\penalty0 315--333, Dec 2006.
\newblock ISSN 1572-9354.
\newblock \doi{10.1007/s10601-006-9003-7}.
\newblock URL \url{https://doi.org/10.1007/s10601-006-9003-7}.

\bibitem[Eisenhandler and Tzur(2019)]{Eisenhandler2019}
Ohad Eisenhandler and Michal Tzur.
\newblock The humanitarian pickup and distribution problem.
\newblock \emph{Operations Research}, 67\penalty0 (1):\penalty0 10--32, 2019.
\newblock \doi{10.1287/opre.2018.1751}.
\newblock URL \url{https://doi.org/10.1287/opre.2018.1751}.

\bibitem[Gamrath et~al.(2020{\natexlab{a}})Gamrath, Anderson, Bestuzheva, Chen,
  Eifler, Gasse, Gemander, Gleixner, Gottwald, Halbig, Hendel, Hojny, Koch,
  Le~Bodic, Maher, Matter, Miltenberger, M{\"u}hmer, M{\"u}ller, Pfetsch,
  Schl{\"o}sser, Serrano, Shinano, Tawfik, Vigerske, Wegscheider, Weninger, and
  Witzig]{GamrathEtal2020ZR}
Gerald Gamrath, Daniel Anderson, Ksenia Bestuzheva, Wei-Kun Chen, Leon Eifler,
  Maxime Gasse, Patrick Gemander, Ambros Gleixner, Leona Gottwald, Katrin
  Halbig, Gregor Hendel, Christopher Hojny, Thorsten Koch, Pierre Le~Bodic,
  Stephen~J. Maher, Frederic Matter, Matthias Miltenberger, Erik M{\"u}hmer,
  Benjamin M{\"u}ller, Marc~E. Pfetsch, Franziska Schl{\"o}sser, Felipe
  Serrano, Yuji Shinano, Christine Tawfik, Stefan Vigerske, Fabian Wegscheider,
  Dieter Weninger, and Jakob Witzig.
\newblock {The SCIP Optimization Suite 7.0}.
\newblock ZIB-Report 20-10, Zuse Institute Berlin, March 2020{\natexlab{a}}.
\newblock URL \url{http://nbn-resolving.de/urn:nbn:de:0297-zib-78023}.

\bibitem[Gamrath et~al.(2020{\natexlab{b}})Gamrath, Anderson, Bestuzheva, Chen,
  Eifler, Gasse, Gemander, Gleixner, Gottwald, Halbig, Hendel, Hojny, Koch,
  Le~Bodic, Maher, Matter, Miltenberger, M{\"u}hmer, M{\"u}ller, Pfetsch,
  Schl{\"o}sser, Serrano, Shinano, Tawfik, Vigerske, Wegscheider, Weninger, and
  Witzig]{GamrathEtal2020OO}
Gerald Gamrath, Daniel Anderson, Ksenia Bestuzheva, Wei-Kun Chen, Leon Eifler,
  Maxime Gasse, Patrick Gemander, Ambros Gleixner, Leona Gottwald, Katrin
  Halbig, Gregor Hendel, Christopher Hojny, Thorsten Koch, Pierre Le~Bodic,
  Stephen~J. Maher, Frederic Matter, Matthias Miltenberger, Erik M{\"u}hmer,
  Benjamin M{\"u}ller, Marc~E. Pfetsch, Franziska Schl{\"o}sser, Felipe
  Serrano, Yuji Shinano, Christine Tawfik, Stefan Vigerske, Fabian Wegscheider,
  Dieter Weninger, and Jakob Witzig.
\newblock {The SCIP Optimization Suite 7.0}.
\newblock Technical report, Optimization Online, March 2020{\natexlab{b}}.
\newblock URL
  \url{http://www.optimization-online.org/DB_HTML/2020/03/7705.html}.

\bibitem[Gendreau et~al.(2001)Gendreau, Laporte, and Semet]{Gendreau2001}
Michel Gendreau, Gilbert Laporte, and Frédéric Semet.
\newblock A dynamic model and parallel tabu search heuristic for real-time
  ambulance relocation.
\newblock \emph{Parallel Computing}, 27\penalty0 (12):\penalty0 1641--1653,
  2001.
\newblock ISSN 0167-8191.
\newblock \doi{https://doi.org/10.1016/S0167-8191(01)00103-X}.
\newblock URL
  \url{http://www.sciencedirect.com/science/article/pii/S016781910100103X}.
\newblock Applications of parallel computing in transportation.

\bibitem[Heier~Stamm et~al.(2017)Heier~Stamm, Serban, Swann, and
  Wortley]{HeierStamm2017}
Jessica~L. Heier~Stamm, Nicoleta Serban, Julie Swann, and Pascale Wortley.
\newblock Quantifying and explaining accessibility with application to the 2009
  h1n1 vaccination campaign.
\newblock \emph{Health Care Management Science}, 20\penalty0 (1):\penalty0
  76--93, Mar 2017.
\newblock ISSN 1572-9389.
\newblock \doi{10.1007/s10729-015-9338-y}.
\newblock URL \url{https://doi.org/10.1007/s10729-015-9338-y}.

\bibitem[Jacobsen(1971)]{jacobsen71}
Stephen Jacobsen.
\newblock On marginal allocation in single constraint min-max problems.
\newblock \emph{Management Science}, 17\penalty0 (11):\penalty0 780--783, 1971.
\newblock \doi{10.1287/mnsc.17.11.780}.
\newblock URL \url{https://doi.org/10.1287/mnsc.17.11.780}.

\bibitem[Katoh et~al.(1985)Katoh, Ibaraki, and Mine]{katoh85}
Naoki Katoh, Toshihide Ibaraki, and H.~Mine.
\newblock An algorithm for the equipollent resource allocation problem.
\newblock \emph{Mathematics of Operations Research}, 10\penalty0 (1):\penalty0
  44--53, 1985.
\newblock URL
  \url{https://EconPapers.repec.org/RePEc:inm:ormoor:v:10:y:1985:i:1:p:44-53}.

\bibitem[Koopman(1953)]{koopman53}
Bernard~O. Koopman.
\newblock The optimum distribution of effort.
\newblock \emph{Journal of the Operations Research Society of America},
  1\penalty0 (2):\penalty0 52--63, 1953.
\newblock \doi{10.1287/opre.1.2.52}.
\newblock URL \url{https://doi.org/10.1287/opre.1.2.52}.

\bibitem[Luss(2012)]{luss2012equitable}
Hanan Luss.
\newblock \emph{Equitable Resource Allocation: Models, Algorithms and
  Applications}.
\newblock Information and Communication Technology Series,. Wiley, 2012.
\newblock ISBN 9781118449219.
\newblock URL \url{https://books.google.ca/books?id=Z2\_3oWjnASkC}.

\bibitem[Margot(2010)]{margot2010symmetry}
Fran{\c{c}}ois Margot.
\newblock Symmetry in integer linear programming.
\newblock In \emph{50 Years of Integer Programming 1958-2008}, pages 647--686.
  Springer, 2010.

\bibitem[Matérn(1960)]{matern}
Bertil Matérn.
\newblock Spatial variation – stochastic models and their applications to
  some problems in forest survey sampling investigations.
\newblock \emph{Report of the Forest Research Institute of Sweden}, 49\penalty0
  (5):\penalty0 1--144, 1960.

\bibitem[Ogryczak et~al.(2014)Ogryczak, Luss, Pi{\'o}ro, Nace, and
  Tomaszewski]{Ogryczak2014-b}
Wlodzimierz Ogryczak, Hanan Luss, Micha{\l} Pi{\'o}ro, Dritan Nace, and Artur
  Tomaszewski.
\newblock Fair optimization and networks: A survey.
\newblock \emph{Journal of Applied Mathematics}, 2014:\penalty0 612018, Sep
  2014.
\newblock ISSN 1110-757X.
\newblock \doi{10.1155/2014/612018}.
\newblock URL \url{https://doi.org/10.1155/2014/612018}.

\bibitem[{Ogryczak, Wlodzimierz}(2014)]{Ogryczak2014}
{Ogryczak, Wlodzimierz}.
\newblock Fair optimization -- methodological foundations of fairness in
  network resource allocation.
\newblock In \emph{2014 IEEE 38th International Computer Software and
  Applications Conference Workshops}, pages 43--48, 2014.

\bibitem[Porteus and Yormark(1972)]{py72}
Evan~L. Porteus and Jonathan~S. Yormark.
\newblock More on min-max allocation.
\newblock \emph{Management Science}, 18\penalty0 (9):\penalty0 502--507, 1972.
\newblock \doi{10.1287/mnsc.18.9.502}.
\newblock URL \url{https://doi.org/10.1287/mnsc.18.9.502}.

\bibitem[Procaccia(2015)]{procaccia2015cake}
Ariel~D Procaccia.
\newblock Cake cutting algorithms.
\newblock In \emph{Handbook of computational social choice, chapter 13}.
  Citeseer, 2015.

\bibitem[Toshihide~Ibaraki(1988)]{resalloc}
Naoki~Katoh Toshihide~Ibaraki.
\newblock \emph{Resource Allocation Problems: Algorithmic Approaches}, page~1.
\newblock Foundations of Computing. The MIT Press, 1988.
\newblock ISBN 0262090279,9780262090278.

\bibitem[Zeitlin(1981)]{ze81}
Zeev Zeitlin.
\newblock Minimization of maximum absolute deviation in integers.
\newblock \emph{Discrete Applied Mathematics}, 3\penalty0 (3):\penalty0 203 --
  220, 1981.
\newblock ISSN 0166-218X.
\newblock \doi{https://doi.org/10.1016/0166-218X(81)90017-2}.
\newblock URL
  \url{http://www.sciencedirect.com/science/article/pii/0166218X81900172}.

\end{thebibliography}

\section*{Appendix}

\appendix

\section{Dual of the CMP}

\begin{subequations}
\begin{alignat}{30}
\max \quad& T\mu &&\\
\textrm{s.t.} \quad 
&\sum_{i=1}^n \alpha_i = 1 &&\quad (g)&&\\
&\sum_{i=1}^n \beta_i = 1 &&\quad(h)&&\\
&\lambda_i - \alpha_i + \beta_i = 0&&\quad(y_i) &&\quad \forall \,i = 1, \dots, n\\
&\mu \leq \frac{1}{T}\sum_{i=1}^n \benefit_{ij}\lambda_i&&\quad (q_j) &&\quad \forall \,j = 1, \dots, k \label{eq:BAPdual:many}\\
&\alpha_i,\,\beta_i \ge 0 && &&\quad \forall \,i = 1, \dots, n.
\end{alignat}
\label{eq:BAPdual}
\end{subequations}

\end{document}